\documentclass[a4paper,11pt]{amsart}

\usepackage{ifpdf}
\usepackage{amsmath, amsthm, amsfonts, amssymb, amscd}
\usepackage[active]{srcltx}
\usepackage{color}
\usepackage[ansinew]{inputenc}

\newtheorem{theorem}{Theorem}[section] 

\newtheorem{lemma}[theorem]{Lemma}

\newtheorem{proposition}[theorem]{Proposition}
\newtheorem*{proposition*}{Proposition}

\newtheorem*{question*}{Question}
\newtheorem*{theorem*}{Theorem}
\newtheorem*{claim*}{Claim}

\newtheorem*{corollary*}{Corollary}

\theoremstyle{definition}

\theoremstyle{remark}

\newtheorem*{remark*}{Remark}
\newtheorem*{definition*}{Definition}

\newcommand{\R}{\mathbb{R}}\newcommand{\N}{\mathbb{N}}
\newcommand{\Z}{\mathbb{Z}}\newcommand{\Q}{\mathbb{Q}}
\newcommand{\T}{\mathbb{T}}\newcommand{\C}{\mathbb{C}}
\newcommand{\D}{\mathbb{D}}\newcommand{\A}{\mathbb{A}}

\begin{document}

\author[P. Le Calvez]{Patrice Le Calvez}
\address{Institut de Math\'ematiques de Jussieu-Paris Rive Gauche, IMJ-PRG, Sorbonne Universit\'e, Universit\'e Paris-Diderot, CNRS, F-75005, Paris, France \enskip \& \enskip Institut Universitaire de France}
\curraddr{}
\email{patrice.le-calvez@imj-prg.fr}

\title[Surface  homeomorphisms with finitely many periodic points]{Conservative surface  homeomorphisms with finitely many periodic points}
\begin{abstract} The goal of the article is to characterize the conservative homeomorphisms of a closed orientable surface $S$
of genus $\geq 2$, that have finitely many periodic points. By conservative, we mean a map with no wandering point. As a particular case, when $S$ is furnished with a symplectic form, we characterize the symplectic diffeomorphisms of $S$ with finitely many periodic points.  \end{abstract}

\bigskip

\maketitle
\tableofcontents
\section{Introduction}

Let $S$ be a closed surface furnished with an area form $\omega$ and its associated Borel measure $\lambda_{\omega}$. What are the simplest examples of diffeomorphisms that preserve $\omega$ (or homeomorphisms that preserve $\lambda_{\omega}$) and that have finitely many periodic points? If $S$ is the $2$-sphere, the irrational rotations provide a natural family of examples. For every $\alpha\in\T=\R/\Z$, denote $R_{\alpha} :\widehat{\C}\to \widehat{\C}$,  the extension of the rotation $z\mapsto e^{2i\pi \tilde\alpha}z$ to the Riemann sphere, where $\tilde\alpha+\Z=\alpha$.  An irrational rotation is a map conjugate to $R_{\alpha}$, where $\alpha\not\in\Q/\Z$. In the case where $S$ is the $2$-torus $\T^2=\R^2/\Z^2$ and $\omega=dx\wedge dy$, there are two natural families of $\omega$-preserving diffeomorphisms with no periodic points: the irrational rotations
$$R_{\beta,\gamma}: (x,y)\mapsto (x+\beta, y+\gamma), \enskip (\beta,\gamma)\not \in{\Q}^2/{\Z}^2,$$ and the skew products over an irrational rotation of the circle
$$S_{\delta, k}: (x,y)\mapsto (x+ky, y+\delta), \enskip \delta\not\in\Q/{\Z}, \enskip k\in{\Z}\setminus\{0\}.$$ 
In the case where the genus of $S$ is larger than $1$, examples can be constructed, conjugate to the time one map of the flow of a minimal direction for a translation surface, giving birth to a much wider family of ``classical examples''. Recall that such flows admit a one dimensional section,  with a return map that is an interval exchange transformation.

One can modify some of the previous examples to enlarge our class of conservative maps with finitely many periodic points. The rotation $R_{\beta,\gamma}$ is the time one map of the flow induced by the constant vector field $X_{\tilde\beta,\tilde\gamma}:  (x,y)\mapsto(\tilde\beta,\tilde\gamma)$, where $\beta=\tilde\beta+\Z$ and $\gamma=\tilde\gamma+\Z$. Suppose that $\tilde\beta$ and $\tilde\gamma$ are rationally independent and let $\psi: \T^2\to[0,+\infty)$ be a continuous function with finitely many zeros. The time one map of the flow induced by the vector field $\psi X_{\tilde\beta,\tilde\gamma}$ has no periodic point except the zeros of $\psi$. Moreover, it preserves the measure $\psi ^{-1}\lambda_{\omega}$. This measure being finite if the function $\psi^{-1}$ is integrable, one can construct in that way  a homeomorphism of the $2$-torus, preserving a finite measure equivalent to the Lebesgue measure $\lambda_{\omega}$, having an arbitrarily large number of fixed points and no other periodic point. There is a well known alternative way to construct a smooth example: in a neighborhood $W$ of a given point, choose a symplectic system of coordinates $(u,v)$, vanishing at this point, such that the vector field can be written $\partial/\partial u$ and is associated to the Hamiltonian function $v$. One can replace this function with a new Hamiltonian function having a unique singular point, being equal to $v(u^2+v^2)$ in a neighborhood of $(0,0)$, giving birth to a foliation (with a singular leaf) such that every new leaf intersects the boundary of $W$ in a pair of points belonging to an old leaf. A similar construction can be done on a surface of genus $\geq 2$ for a $\omega$-preserving vector field to obtain a smooth symplectic map with an arbitrarily large number of fixed points and no other periodic point. Note that the map $f$ obtained in that way is isotopic to the identity relative to its fixed point set. Consider now a finite cyclic covering $\hat S$ of $S$. There is a unique lift $\hat f$ of $f$ to $\hat S$ that is isotopic to the identity. For every covering automorphism $T$,  the map $T\circ \hat f$ preserves the lifted form $\hat \omega$ and has finitely many periodic points, all of them with the same period.

Beyond the previous examples, one can find symplectic maps having finitely many periodic points, with a much richer dynamics. The approximation method by conjugation, introduced by Anosov-Katok, permits to construct smooth symplectic diffeomorphisms on the $2$-sphere with exactly two fixed points $z_0$, $z_1$  and no other periodic point, such that $\lambda_{\omega}$ is ergodic (see \cite {AK}).  In these examples, there exists an irrational number $\tilde\alpha$ (in fact a Liouville number) such that every point of $S\setminus \{z_0,z_1\}$ ``turns''  in the annulus $S\setminus \{z_0,z_1\}$ with an angular speed equal to $2\pi\tilde\alpha$. In fact, for every Liouville number, one can make the construction in such a way that $\lambda_{\omega}$ is weakly mixing (see \cite {FS}). Examples of symplectic maps having finitely many periodic points, with a rich dynamics can also be constructed in the $2$-torus (see \cite{WZ} for instance).

\medskip
The main issue of this article is to understand to what extent the examples given above describe all symplectic diffeomorphisms of surfaces that have finitely many periodic points. We will see that they permit to get a classification: every symplectic diffeomorphism with finitely many periodic points is naturally associated to one of the examples above. Moreover, this classification is valid for a wider set: the set of homeomorphisms with no wandering point (that we will call {\it non wandering homeomorphisms}) Furthermore we will see that non wandering homeomorphisms with infinite many periodic points can be divided into maps of finite order and maps with periodic points of arbitrarily large period.  Before stating the precise results, recall that a non empty open set $U\subset S$ is wandering if $U\cap f^{-n}(U)=\emptyset$ for every $n\geq 1$ and that points of $U$ are wandering points. So, $f$ is non wandering if, for every non empty open set $U\subset S$, there exists $n\geq 1$ such that $U\cap f^{-n}(U)\not=\emptyset$. By Poincar\'e's Recurrence Theorem, every symplectic diffeomorphism, or more generally every $\lambda_{\omega}$-preserving homeomorphism, is non wandering. 

\subsection{The case of the sphere}

 Denote  $ \pi:(r,\theta)\mapsto re^{2i\pi\theta}$ the covering projection defined on the universal covering space $(0,+\infty)\times\R$ of $\C\setminus\{0\}$. An {\it irrational pseudo-rotation of $\C$ of rotation number $\alpha\in\T\setminus\Q/\Z$} is a non wandering homeomorphism $f$ of $\C$ that fixes $0$ and that satisfies the following property:
 
  if $\widetilde f$ is a lift of $f\vert_{\C\setminus\{0\}}$ to $(0,+\infty)\times\R$, there exists $\tilde\alpha\in\R$ satisfying $\tilde\alpha+\Z=\alpha$ such that for every compact set $\Xi\subset\C\setminus\{0\}$, and every $\varepsilon >0$, there exists $N\geq 1$ such that $$(P_{\tilde \alpha}):\enskip n\geq N \enskip\mathrm{and}\enskip \widetilde z\in \pi^{-1}(\Xi)\cap\widetilde f^{-n}(\pi^{-1}(\Xi))\Rightarrow \left\vert {p_1(\widetilde f^n(\widetilde z))-p_1(\widetilde z)\over n}-\tilde\alpha\right\vert \leq \varepsilon,$$ where $p_1: (\theta,r)\mapsto \theta$ is the projection on the first factor.\footnote{Requiring the whole set $\Xi=\C\setminus\{0\}$ to satisfy $(P_{\tilde \alpha})$ would be too strong, and in opposition to this weaker definition, would define a property that is not  invariant by conjugacy in the group of orientation preserving homeomorphisms of $\C$ that fix $0$.} 
 By extension, every homeomorphism of a $2$-sphere that is conjugate to the extension of $f$ to $\hat {\C}$ would be called an irrational pseudo-rotation of rotation number $\alpha$. 

The following result gives a very precise description of non wandering homeomorphisms of the $2$-sphere with finitely many periodic points.

\begin{theorem} \label{th:rappel1} Let $f$ be an orientation preserving and non wandering homeomorphism of the $2$-sphere. Then exactly one of the following assertions holds:

\begin{enumerate}
\item The map $f$ has periodic points of period arbitrarily large.
\item There exists $\alpha\in\Q/\Z$ such that $f$ is conjugate to $R_{\alpha}$.
\item There exists $\alpha\in \T\setminus\Q/\Z$ such that $f$ is an irrational pseudo-rotation of rotation number $\alpha$. 

\end{enumerate}
\end{theorem}

Theorem \ref{th:rappel1} is well known. Let us briefly explain why it is true. It is known (see Franks \cite{F4}) that an area preserving homeomorphism of the $2$-sphere that has at least three fixed points has infinitely many periodic points (with periods arbitrarily large if the map has infinite order). Moreover, this result is also true for non wandering homeomorphisms (see \cite{Lec2}). So, if $f$ is an orientation preserving and non wandering homeomorphism of the $2$-sphere and if neither (1) nor (2) holds, then $f$ has no periodic point but two fixed points. To get (3) it remains to study the non wandering homeomorphisms of the annulus $\T\times\R$ that are isotopic to the identity (this means that the orientation is preserved and the ends are fixed) and that have no periodic point. Assertion (3) is related to the Poincar\'e-Birkhoff Theorem and its many generalizations, its meaning is that there is a unique rotation number (for whatever reasonable definition)  and moreover that it cannot be rational (see \cite{F1}, \cite{Lec1}).

\subsection{The case of the torus} A classification of the area preserving homeomorphisms of $\T^2$ with finitely many periodic points has been done by Addas-Zanata and Tal \cite{AT}. This classification is still valid for non wandering homeomorphisms. We will state the result here but will give the definitions appearing in the statement  in the last section of this article (rotation set, vertical rotation set, automorphism of $\T^2$). We will give the proof in the same section. The proof is nothing but the original proof of \cite{AT} except at one point where a later result of Addas-Zanata, Garcia and Tal  \cite{AGT} is needed to replace the area preserving condition with the non wandering condition.

\begin{theorem} \label{th:rappel2}  Let $f$ be an orientation preserving  and non wandering homeomorphism of $\T^2$. Then exactly one of the following assertions holds:

\begin{enumerate}
\item The map $f$ has periodic points of period arbitrarily large.
\item There exist $g\in \mathrm{Aut}(\T^2)$, $k\in\Z\setminus\{0\}$ and $\delta\not\in \Q/\Z$ such that $g\circ f\circ g^{-1}$ is isotopic to $S_{0,k}$, with a vertical rotation set reduced to $\delta$. In this case $f$ has no periodic point.
\item The map $f$ is isotopic to the identity and its rotation set (for every lift) is a point or a segment that does not meet $\Q^2/\Z^2$. In this case $f$ has no periodic point.
\item There exists an integer $q\geq 1$ such that: 

\begin{itemize}
\item the periodic points of $f^q$ are fixed;
\item the fixed point set of $f^q$ is non empty and $f^q$ is isotopic to the identity relative to it;
\item the rotation set of the lift of $f^q$ that has fixed points is reduced to $0$ or is a segment with irrational slope that has zero as an end point.\footnote{This last case contains the case where $f$ has finite order}
\end{itemize}

\end{enumerate}
\end{theorem}

\subsection{The case of high genus} 
Let us state the main result of the article that gives a characterization of non wandering homeomorphisms of a surface of genus $\geq 2$ that have finitely many periodic points:

\begin{theorem} \label{th:main}  Let $S$ be an orientable closed surface of genus $g\geq2$ and $f$ an orientation preserving and non wandering homeomorphism of $S$. Then the following alternative holds:

\begin{enumerate}
\item The map $f$ has periodic points of period arbitrarily large.
\item There exists an integer $q\geq 1$ such that: 

\begin{itemize}
\item the periodic points of $f^q$ are fixed;
\item  the fixed point set of $f^q$ is non empty and $f^q$ is isotopic to the identity relative to it . \footnote{Here again the second case contains the case where $f$ has finite order}

\end{itemize}

\end{enumerate}
\end{theorem}

\begin{remark*} Suppose that $S$ is furnished with a symplectic form $\omega$ and that $f$ preserves $\lambda_{\omega}$. If $f$ satisfies (2) and if $f^q$ is not the identity, then the rotation vector of $\lambda_{\omega}$ (for $f^q$) is a non zero element of $H_1(S,\R)$ (in other words $f^q$ is not Hamiltonian). This is a particular case of the Conley Conjecture (see \cite{FH} or \cite{Lec2}). The examples we must have in mind are the ones given at the beginning of the introduction where a section exists.
\end{remark*}

Let us explain the structure of the proof of Theorem \ref{th:main}.  

\begin{definition*} Let $S$ be an orientable closed surface of genus $\geq 2$. 
A {\it Dehn twist map} of $S$ is an orientation preserving homeomorphism $h$ of $S$ that satisfies the following properties:

\begin{itemize}

\item there exists a non empty finite family $(A_{i})_{i\in I}$ of pairwise disjoint invariant essential closed annuli (meaning sets homeomorphic to $\T\times[0,1]$ with boundary loops non homotopic to zero in $S$); 

\item no connected component of $S\setminus \cup_{i\in I}A_i$ is an annulus (meaning a set homeomorphic to $\T\times(0,1)$);

\item $h$ fixes every point of $S\setminus \cup_{i\in I}A_i$;

\item for every $i\in I$, there exists $r_i\in\Z\setminus\{0\}$ such that  $h_{\vert A_i}$ is conjugate to $\tau^{r_i}$, where $\tau$ is the homeomorphism of $\T\times[0,1]$ that is lifted to the universal covering space $\R\times[0,1]$ by $\tilde\tau:(x,y)\mapsto (x+y,y)$.

\end{itemize}

The annuli $A_i$ will be called  the {\it twisted annuli} and $r_i$ the {\it twist coefficients}.
\end{definition*}

The Nielsen-Thurston Decomposition Theorem (see \cite{CB}, \cite{FLP} or \cite{T}) tells us the following:  if $f$ is an orientation preserving homeomorphism of $S$, then there exists a finite family $(A_{i})_{i\in I}$ of pairwise disjoint essential closed annuli and a homeomorphism $f_*$ isotopic to $f$ such that:

\begin{itemize}

\item no connected component of $S\setminus \cup_{i\in I}A_i$ is an annulus;

\item the family $(A_i)_{i\in I}$ is invariant by $f_*$;

\item  for every connected component $W$ of $S\setminus \cup_{i\in I}A_i$, there exists $q$ such that $f_*{}^q(W)=W$ and $f_*{}^q{}_{\vert W}$ is isotopic to a map of finite order or to a pseudo-Anosov map;

\item for every $i\in I$, there exists $q_i$ such that $f_*{}^{q_i}(A_i)=A_i$ and $k_i\in \Z$ such that  $f_*{}^{q_i}{}_{\vert A_i}$ is conjugate to $\tau^{k_i}$.
\end{itemize}

In particular, the following classification holds for an orientation preserving homeomorphism $f$ of $S$:

\begin{itemize}

\item there exists at least one component of pseudo-Anosov type in the Nielsen-Thurston classification;

\item there exists $q\geq 1$ such that $f^q$ is isotopic to a Dehn twist map;

\item  there exists $q\geq 1$ such that $f^q$ is isotopic to the identity.
\end{itemize}

It is well known and folklore that if there exists a component of pseudo-Anosov type, then $f$ has infinitely periodic points of arbitrarily large period (see \cite{FLP} or \cite{Ha}).  As we will see in the next section, it is easy to prove that a power of a non wandering map is still non wandering. So, Theorem \ref{th:main}  can be deduced from the two following results:

\begin{proposition} \label{prop:twist} Let $S$ be an orientable closed surface of genus $g\geq2$ and $f$ a non wandering homeomorphism of $S$. If $f$ is isotopic to a Dehn twist map, then it has periodic points of period arbitrarily large.
\end{proposition}

\begin{proposition} \label{prop:connexions} Let $S$ be an orientable closed surface of genus $g\geq2$ and $f$ a non wandering homeomorphism of $S$. If $f$ is isotopic to the identity, then $f$ has fixed points and
\begin{itemize}
\item either $f$ has periodic points of period arbitrarily large;

\item or every periodic point of $f$ is fixed and $f$ is isotopic to the identity relative to its fixed point set.
\end{itemize} 

\end{proposition}

\subsection{Plan of the article} 

Proposition \ref{prop:twist} is a kind of Poincar\'e-Birkhoff Theorem in surfaces of high genus. Its proof will be given in Section \ref{s.dehn}.  Proposition \ref{prop:connexions} tells us that, roughly speaking, a non wandering homeomorphism with finitely many periodic points is ``modelized'' by the time one map of the flow of a minimal direction for a translation surface, after adding stopping points and lifting to a cyclic finite covering.  Its proof will be given in Section \ref{s.isoidentty}. The proof of Theorem  \ref{th:rappel2}  will be given in Section \ref{s.torus}. A certain number of definitions and preliminary results will be given in Section \ref{s:preliminaries}, most of the results being well known. However we will state two ``new technical results'' in this section, a fixed point theorem and a forcing result, which are inspired by common works with Fabio Tal (\cite{LT1} and \cite{LT2}).

Section  \ref{s.dehn} and Section \ref{s.isoidentty} are the more technical parts of the article.  Surprisingly, the proofs of Proposition \ref{prop:twist}  and  \ref{prop:connexions}  are very similar. Let us give more details about the proof of Proposition  \ref{prop:twist}. In a recent  work with Mart\'in Sambarino \cite{LS}, we have proved that a generic symplectic diffeomorphism of a surface of genus $\geq 2$, for the $C^k$-topology, $k\geq 1$, has transverse homoclinic intersections. A crucial argument is to prove that a generic symplectic diffeomorphism of a surface of genus $\geq 2$ has more than $2g-2$ periodic points. The case of diffeomorphisms isotopic to a Dehn twist is studied in a long section of the article. We note that such a diffeomorphism $f$ has a lift $\hat f$ to a certain annular covering space $\hat S$ that satisfies a ``twist condition'' and so, if $f$ has finitely many periodic points, then $\hat f$ cannot satisfy the ``intersection property'': there exists an essential simple loop $\hat \lambda\subset\hat S$ that is disjoint from its image by $\hat f$. Lifting $\hat f$ to a diffeomorphism $\tilde f$ of the universal covering space $\tilde S$, looking at the action of the dynamics of $\tilde f$ on the set of lifts of $\hat \lambda$ and using the properties of the stable and unstable manifolds of the fixed saddle points (there exists at least $2g-2$ such points)  we succeed to prove that homoclinic intersections exist. What is done in the present article is to show that the existence of infinitely many periodic points can be obtained without using the saddle points. Looking at the dynamics of $\tilde f$ on the set of lifts of $\hat \lambda$  is sufficient if we use the fixed point theorem and the forcing result stated in Section \ref{s:preliminaries}. In particular the proof is  valid for area preserving homeomorphisms, or more generally for non wandering homeomorphisms. 

To conclude this introduction, let us state some other known results about the dynamics of a non wandering homeomorphism $f$ with finitely many periodic points on a surface $S$ of genus $\geq 2$.  Theorem \ref{th:main} tells us that it is sufficient to look at the case where $f$ is isotopic to the identity relative to its fixed point set and has no other periodic point. A result of Lellouch \cite{Lel} says that if $\mu_1$ and $\mu_2$ are two invariant Borel probability measures, then the rotation vectors of $\mu_1$ and $\mu_2$ do not intersect (for the canonical intersection form $\wedge$ on $H_1(S,\R)$). Another result, that can be found in \cite{LS}, is the existence, in the case where $f$ has finitely many fixed points, of a section in the following sense: there exists a simple oriented loop $\lambda\subset S$ non homologous to zero, such that if $\check S$ is the infinite cyclic covering space of $S$ associated to $\lambda$ and $\check f$ the natural lift of $f$ to $\check S$, then for every loop $\check \lambda \subset \hat S$ that lifts $\lambda$, the points that are on $\check \lambda$ and not fixed by $\check f$ are sent by $\check f$ on the left of $\check \lambda$ and by $\check f^{-1}$ on its right. It would be a natural challenge to look for further dynamical properties of these maps (what looks like the rotation set? does there always exist a section if $f$ is not the identity?)

\section{Definitions, notations and preliminaries} \label{s:preliminaries}

\subsection{Loops and paths}  Let $S$ be an orientable connected surface (not necessarily closed, not necessarily boundaryless).  A {\it loop} of $S$ is a continuous map $\gamma: \T\to S$.  It will be called {\it essential} if it is not homotopic to a constant loop. A {\it path} of $S$ is a continuous map $\gamma: I\to S$ where $I\subset \R$ is an interval. A loop or a path will be called 
{\it simple} if it is injective. A {\it segment} is a simple path $\sigma:[a,b]\to X$, where $a<b$. The points $\sigma(a)$ and $\sigma(b)$ are the {\it ends} of $\sigma$ and the set $\sigma((a,b))$ its {\it interior}. We will say that $\sigma$ {\it joins} $\sigma(a)$ to $\sigma(b)$. More generally if $A$ and $B$ are disjoints, we will say that $\sigma$ joins $A$ to $B$, if $\sigma(a)\in A$ and $\sigma(b)\in B$. A {\it line} is a proper simple path $\lambda:\R\to S$, a half line a proper simple path $\lambda:I\to S$, where $I=[a,+\infty)$ or $I=(-\infty,a]$. In that case, $\gamma(a)$ is its end. As it is usually done we will use the same name and the same notation to refer to the image of a loop or a path $\gamma$. 

Note that a simple loop or a simple path is naturally oriented. If $\gamma$ is a simple loop that separates  $S$ (meaning that its complement has two connected components) the one that is located on the right of $\gamma$ will be  denoted $R(\gamma)$ and the other one $L(\gamma)$. We will use the same notation for a line that separates $S$, in particular for a line of $\R^2$.

 Let $f$ be an orientation preserving homeomorphism of $\R^2$.  A {\it Brouwer line} of $f$ is a line $\lambda\subset\R^2$ such that $f(\lambda)\subset L(\lambda)$ and $f^{-1}(\lambda)\subset  R(\lambda)$. Equivalently it means that $f(\overline {L(\lambda)})\subset L(\lambda)$ or that $f^{-1}(\overline {R(\lambda)})\subset R(\lambda)$.

\subsection{Homeomorphisms of hyperbolic surfaces}

Let $S$ be a connected closed  orientable  surface of genus $g\geq2$.
Furnishing $S$ with a Riemannian metric of constant negative curvature, one can suppose that the universal covering space of $S$ is the disk $\D=\{z\in \C\,\vert\, \vert z\vert <1\}$ and that the group of covering transformations, denoted $G$, is composed of  M\H obius automorphisms of $\D$. Every element $T\in G$ is hyperbolic: it can be extended to a homeomorphism of $\overline{\D}=\{z\in \C\,\vert\, \vert z\vert \leq1\}$ having two fixed points, both on the boundary: a repelling fixed point $\alpha(T)$ and an attracting fixed point $\omega(T)$. For every $z\in \overline{\D}\setminus\{\alpha(T), \omega(T)\}$, it holds that
 $$\lim_{k\to -\infty} T^k(z) =\alpha(T), \lim_{k\to +\infty} T^k(z) =\omega(T).$$
 We define a {\it $T$-line} to be a line of $\D$ invariant by $T$ and oriented in such a way that it can be extended to a segment of $\overline{\D}$ that joins $\alpha(T)$ to $\omega(T)$. 
 
 It is well known that a homeomorphism $\tilde f$ of $\D$ that lifts a homeomorphism $f$ of $S$ can be extended to a homeomorphism of  $\overline{\D}=\{z\in \C\,\vert\, \vert z\vert \leq1\}$. If $[\tilde f]$ is the automorphism of $G$ defined by the relation:
 $$\tilde f(T(z))=[\tilde f](T)(\tilde f(z)), \enskip \mathrm{for \enskip all} \enskip z\in \D,$$ then the extension of $\tilde f$ satisfies
 $$\tilde f(\alpha(T))=\alpha([\tilde f](T)) \enskip \mathrm{and}\enskip  \tilde f(\omega(T))=\omega([\tilde f](T)) \enskip  \enskip \mathrm{for \enskip all} \enskip T\in G.$$ 
 If $h$ is a homeomorphism of $S$ that is isotopic to $f$, then every isotopy from $f$ to $h$ can be lifted to an isotopy from $\tilde f$ to a certain lift $\tilde h$ of $h$. The lift $\tilde h$ does not depend on the initial isotopy and has the same extension as $\tilde f$ on the boundary circle, because the automorphisms $[\tilde f] $ and $[\tilde h]$ coincide. For conveniency, we will write $\tilde S$ for the universal covering space of $S$ and $\partial \tilde S$ for its boundary defined via the identification of $\tilde S$ with $\D$. We will write $\overline{\tilde S}$=$\tilde S\cup \partial(\tilde S)$.
 
  \subsection{Non wandering homeomorphisms}\label{s:NWH}
  Let us recall some very classical easy results that we will use in the article. Recall that if $f$ is a homeomorphism of a surface $S$, a  point $z\in S$ is recurrent if there exists a subsequence of $(f^n(z))_{n\geq 0}$ that converges to $z$. 
  
  \begin{lemma}  \label{l:NWH} Let $f$ be a non wandering homeomorphism of a surface $S$. For every non empty open set $U$ and every $q\geq 1$, there exists an increasing sequence $(n_i)_{0\leq i\leq q}$ in $\N$, satisfying $n_0=0$, such that $\bigcap_{0\leq i\leq q} f^{-n_i}(U)\not=\emptyset$.
    \end{lemma} 
  
  \begin{proof} Let us prove the lemma by induction. By definition of a non wandering homeomorphism, the lemma is true for $q=1$. Suppose that the lemma is true for every $q'<q$, where $q\geq 2$. Let $U\subset S$ be a non empty open set. There exists an increasing sequence $(n_i)_{0\leq i<q}$ in $\N$, satisfying $n_0=0$, such that $\bigcap_{0\leq i<q} f^{-n_i}(U)\not=\emptyset$. 
 As $V=\bigcap_{0\leq i<q} f^{-n_i}(U)$ is open and non empty, there exists $n>0$ such that $V\cap f^{-n}(V)\not=\emptyset$. In particular, it holds that $\bigcap_{0\leq i<q} f^{-n_i}(U)\not=\emptyset$, where $n_q= n_{q-1}+n$. So, the lemma is true for $q$.
 \end{proof}

  \begin{proposition} \label{prop:NWH}  Let $f$ be a non wandering homeomorphism of a surface $S$. Then:
  
  \begin{enumerate}

  \item every power $f^k$, $k\in\Z$, is non wandering;
  
  \item if $\hat S$ is a finite covering of  $S$, every lift of $f$ to $\hat S$ is non wandering;
  
  \item the set of recurrent points is a dense $G_{\delta}$ set.
  
  \end{enumerate}
  \end{proposition} 
  
  \begin{proof} By definition, $f$ is non wandering if and only $f^{-1}$ is non wandering. Moreover the identity is non wandering. So, to prove (1) it is sufficient to prove it for $k\geq 2$. Let $U\subset S$ be a non empty open set. By Lemma \ref {l:NWH}, there exists an increasing sequence $(n_i)_{0\leq i\leq k}$ in $\N$, satisfying $n_0=0$, such that $\bigcap_{0\leq i\leq k} f^{-n_i}(U)\not=\emptyset$. There exists $i_0<i_1$ such that $n_{i_1}-n_{i_0}\in k\Z$. Write $n_{i_1}-n_{i_0}=nk$, where $n>0$. It holds that
  $U\cap f^{-nk} (U)=f^{n_{i_0}}(f^{-n_{i_0}}(U)\cap f^{-n_{i_1}}(U))\not=\emptyset$. So $U$ is a non wandering open set of $f^k$.
  
  Let us prove (2). Suppose that $\hat S$ is a $r$-cover of $S$ and denote $\hat \pi:  \hat S\to S$ the covering projection. To prove that $\hat f$ is non wandering, it is sufficient to prove that if $U\subset S$ is an open disk such that $\hat\pi^{-1}(U)=\bigsqcup_{1\leq j\leq r} \hat U_j$, where each $\hat U_j$ is mapped homeomorphically onto $U$ by $\hat \pi$, then every  $\hat U_j$ is non wandering. By Lemma \ref {l:NWH}, there exists an increasing sequence $(n_i)_{0\leq i\leq r!}$ in $\N$, satisfying $n_0=0$, such that $\bigcap_{0\leq i\leq r!} f^{-n_i}(U)\not=\emptyset$. Let us choose $z\in \bigcap_{0\leq i\leq r!} f^{-n_i}(U)$ and denote $\hat z_j$ the preimage of $z$ belonging to $\hat U_j$. For every $i\in\{0,\dots, r!\}$ denote $\sigma_i\in\mathcal {S}_r$ the permutation such that $\hat f^{n_i}(\hat z_j)\in \hat U_{\sigma_i(j)}$. There exists $i_0<i_1$ such that $\sigma_{i_0}=\sigma_{i_1}$. Setting $n_{i_1}-n_{i_0}=n>0$, one deduces that $\hat U_j\cap \hat f^{-n}(\hat U_j)\not=\emptyset$, for every $j\in \{1,\dots, r\}$. 
  
  To prove (3) furnish $S$ with a distance $d$. If $m\geq 1$ and $q\geq 1$, define
  $$O_{m,q}=\{x \in S, \enskip \exists n\geq q, \enskip d(f^n(x), x)<{1\over m}\}.$$ Applying Lemma  \ref{l:NWH} to a ball $B(z,\varepsilon)$, $\varepsilon<1/2m$, and to $q$, we obtain that $O_{m,q}\cap B(z,\varepsilon)\not=\emptyset$. So $O_{m,q}$ is a dense open set. 
  Noting that $\bigcap _{n\geq 1, q\geq 1} O_{m,q}$ is the set of recurrent points and that $S$ is a Baire space, we can conclude.
  \end{proof}

We will often use the following result:

\begin{proposition} \label{prop:wandering} Let $f$ be a non wandering homeomorphism of a surface $S$ and $\hat f$ a lift of $f$ to a covering space $\hat S$.  If $\hat U$ is a non empty wandering open set of $\hat f$, then $\bigcup_{n\geq 0}\hat f^n(\hat U)$ and $\bigcup_{n\geq 0}\hat f^{-n}(\hat U)$ are not relatively compact.

  \end{proposition}

\begin{proof} Of course it is sufficient to prove that $\bigcup_{n\geq 0}\hat f^n(\hat U)$ is not relatively compact. We will argue by contradiction and suppose that $\bigcup_{n\geq 0}\hat f^n(\hat U)$ is relatively compact. The frontier $\mathrm{fr}\left(\bigcup_{n\geq 0}\hat f^n(\hat U)\right)$ of $\bigcup_{n\geq 0}\hat f^n(\hat U)$ is a compact set with empty interior. The covering map $\hat\pi:\hat S\to S$ being continuous, $\hat\pi\left (\mathrm{fr}\left(\bigcup_{n\geq 0}\hat f^n(\hat U)\right)\right)$ is compact. The map $\hat\pi$ being a local homeomorphism, $\hat\pi\left (\mathrm{fr}\left(\bigcup_{n\geq 0}\hat f^n(\hat U)\right)\right)$ has empty interior.  Indeed, one can cover $\mathrm{fr}\left(\bigcup_{n\geq 0}\hat f^n(\hat U)\right)$ with finitely many open sets that are sent homeomorphically  by $\hat \pi$ on their image. The map $\hat \pi$, being continuous and open, $\hat\pi^{-1} \left(\hat\pi\left (\mathrm{fr}\left(\bigcup_{n\geq 0}\hat f^n(\hat U)\right)\right)\right)$ is a closed set of $\hat S$ with empty interior.  Denote $\mathrm{rec}(f)$ the set of recurrent points of $f$, which is a dense $G_{\delta}$ set. The map $\hat \pi$ being continuous and open, $\hat\pi^{-1}(S\setminus \mathrm{rec}(f))$ is a $F_{\sigma}$ set with empty interior. The contradiction comes from the fact that $\hat U$ is contained in the union of  $\hat\pi^{-1} \left(\hat\pi\left (\mathrm{fr}\left(\bigcup_{n\geq 0}\hat f^n(\hat U)\right)\right)\right)$ and $\hat\pi^{-1}(S\setminus \mathrm{rec}(f))$. Indeed, suppose that the image by $\hat \pi$ of $\hat z\in U$ is recurrent. The closure of  $\bigcup_{n\geq 0}\hat f^n(\hat U)$ contains finitely many preimages of $\hat\pi(\hat z)$. At least one of them belongs to $\omega(\hat z)$, meaning that  it is a limit of  a subsequence of $(\hat f^n(\hat z))_{n\geq 0}$. It cannot belong to $\bigcup_{n\geq 0}\hat f^n(\hat U)$ because $\hat U$ is wandering  and so is on the frontier.  \footnote{In the case where $\hat S$ is a normal covering space and $\hat f$ commutes with the covering automorphisms, a much simpler proof can be given. Unfortunately, this will not be always the case when we will apply  Proposition \ref{prop:wandering}.}

\end{proof}

\begin{remark*} In case $f$ preserves a totally supported finite measure $\mu$, the results above are obvious. Indeed, in Proposition \ref{prop:NWH}, the maps $f^k$ preserve $\mu$ and $\hat f$ preserves the lift of $\mu$ that is a totally supported finite measure. Moreover, as a consequence of Poincar\'e's  Recurrence Theorem, it is known that almost every point is recurrent, and so the set of recurrent points is dense. In Proposition \ref{prop:wandering}, $\hat f$ preserves a totally supported locally finite measure.
\end{remark*}

 \subsection{Poincar\'e-Birkhoff Theorem }\label{s:PB}

\medskip
 We suppose in this subsection that $I$ is a non trivial interval of $\R$. We define the annuli $\A=\T\times I$ and $\mathrm{int}(\A)$, where $\mathrm{int}(\A)$ is obtained from $\A$ by taking out the possible boundary circles. Writing $\tilde \A=\R\times I$ for the universal covering space of $\A$,  we define the covering projection  $\pi: \tilde\A\to \A,\enskip  (x,y)\mapsto (x+\Z,y) $ and the generating covering automorphism $T:\tilde\A\to \tilde \A,\enskip (x,y)\mapsto (x+1,y)$ .

Let $f$ be a homeomorphism of $\A$  isotopic to the identity (meaning orientation preserving and fixing the possible ends or boundary circles) and $\tilde f$ a lift of $f$ to $\tilde \A$. Let us recall the definition of the rotation number of a compactly supported Borel probability measure $\mu$ invariant by $f$. Denote $p_1: \tilde \A\to\R$ the projection on the first factor. The maps $\tilde f$ and $T$ commute, so $p_1\circ \tilde f-p_1$ lifts a continuous function $\psi_{\tilde f} : \A\to\R$.  The {\it rotation number} $ \mathrm{rot}_{\tilde f} (\mu)=\int_{\A} \psi_{\tilde f} \, d\mu\in\R$ measures the mean horizontal displacement of $\tilde f$. 
 
If $z$ is a periodic point of $f$ of period $q$ and if $\tilde z$ is a lift of $z$ in $\tilde \A$, then there exists an integer $p$, independent of the choice of $\tilde z$, such that $\tilde f^q(\tilde z)=T^p(\tilde z)$. We will say that $p/q$ is the rotation number of $z$ for the lift $\tilde f$, it coincides with the rotation vector of the equidistributed measure supported on the orbit of $z$. 
 
We will use many times the following extension of the classical Poincar\'e-Birkhoff Theorem (see \cite{Lec1}):

\begin{theorem} \label{th:PB} Let $f$ be a homeomorphism of $\A$  isotopic to the identity and $\tilde f$ a lift of $f$ to $\tilde \A$. We suppose that there exist two invariant ergodic compactly supported Borel probability measures $\mu_1$ and $\mu_2$, such that  $\mathrm{rot}_{\tilde f} (\mu_1)<\mathrm{rot}_{\tilde f} (\mu_2)$. Then: 
\begin{itemize} 
\item either, for every rational number $p/q\in (\mathrm{rot}_{\tilde f} (\mu_1),\mathrm{rot}_{\tilde f} (\mu_2))$, written in an irreducible way, there exists a periodic point $z$ of $f$ of period $q$ and rotation number $p/q$ for $\tilde f$;
\item or there exists an essential simple loop $\lambda\in \mathrm{int}(\A)$ such that $f(\lambda)\cap\lambda=\emptyset$.

\end{itemize}

\end{theorem}

\subsection{A fixed point theorem for a planar homeomorphism}\label{s.criterion} 

In this sub-section, we will give a criterion of existence of a fixed point for a planar homeomorphism, which is a slight generalization of a result proved  in \cite{LT2}. It will be an essential tool in the proofs of Propositions  \ref{prop:twist}  and  \ref{prop:connexions}.

 Let $(\lambda_i)_{1\leq i\leq r}$ be a finite family of pairwise disjoint lines of $\R^2$. Say that the family is {\it cyclically ordered} if:

\begin{itemize}
\item one can choose, for every $i\in\{1,\dots, r\}$, a connected component $E_i$ of $\R^2\setminus \lambda_i$  in such a way that the sets $\overline E_i$ are pairwise disjoint;

\item for every $i\in\{1,\dots, r-1\}$, one can find two disjoint connected sets, the first one containing $\lambda_i$ and $\lambda_{i+1}$, the other one containing every line $\lambda_j$, $j\not\in\{i,i+1\}$.
\end{itemize}

Note that the complement of $\bigcup_{1\leq i\leq r}E_i$ is a connected sub-surface $\Sigma$ whose boundary is equal to $\bigcup_{1\leq i\leq r}\lambda_i$. Note also that one can find two disjoint connected sets, the first one containing $\lambda_1$ and $\lambda_{r}$, the other one containing every line $\lambda_j$, $1<j<r$. By the extension of Schoenflies Theorem due to Homma (see \cite{Ho}), one knows that $(\lambda_i)_{1\leq i\leq r}$ is cyclically ordered if and only if there exists a homeomorphism of $\R^2$ that sends $\lambda_i$ on the graph of the fonction 
 $$\psi_i:(2i-1, 2i+1)\to\R, \enskip x\mapsto{1\over 1-(x-2i)^2}.$$

\begin{proposition} \label{prop:fixedpoint} Let $f$ be a homeomorphism of $\R^2$. Suppose that there exists a cyclically ordered family $(\lambda_i)_{1\leq i\leq 4}$ of pairwise disjoint lines and for every $i\in\{1,3\}$ a segment $\sigma_i\subset \lambda_i$ such that: 
\begin{itemize}
\item $f(\sigma_i)\cap \lambda_i=\emptyset$  if $i\in\{1,3\}$;
\item $f(\lambda_j)\cap \lambda_j=\emptyset$  if $j\in \{2,4\}$;
\item  $f(\sigma_i)\cap \lambda_j\not=\emptyset$  if $i\in\{1,3\}$ and $j\in \{2,4\}$.

\end{itemize}
Then $f$ has a fixed point.
\end{proposition}

\begin{proof} Recall that $E_i$ is the connected component of $\R^2\setminus\lambda_i$ that does not contain the $\lambda_j$, $ j\not=i$. Taking a sub-segment of $\sigma_1$ if necessary, one can suppose that one end of $f(\sigma_1)$ belongs to $\lambda_2$, the other end to $\lambda_4$ and the other points of  $f(\sigma_1)$ neither to $\lambda_2$ nor to $\lambda_4$. The segment $f(\sigma_1)$, being disjoint from $\lambda_1$, belongs to the connected component of $\R^2\setminus \lambda_1$ that contains $\lambda_2$ and $\lambda_4$, it does not meet $\overline E_1$. Moreover, the interior of $f(\sigma_1)$ is contained in the connected component of $\R^2\setminus (\lambda_2\cup\lambda_4)$ that contains $\lambda_1$ and $\lambda_3$, it does not meet neither $\overline E_2$, nor $\overline E_4$.  One can suppose that  $\sigma_3$ satisfies similar  properties and of course  $f(\sigma_1)$ and $f(\sigma_3)$ are disjoint. Let $\sigma_2\subset\lambda_2$ be the segment that joins the two ends of $f(\sigma_1)$ and $f(\sigma_3)$ that are on $\lambda_2$. By hypothesis, $f^{-1}(\sigma_2)$ is disjoint from $\lambda_2$ and so is included in the connected component of $\R^2\setminus \lambda_2$ that contains $\lambda_1$ and $\lambda_3$, it is disjoint from $\overline E_2$.  The segment $\sigma_4\subset\lambda_4$ that joins the two ends of $f(\sigma_1)$ and $f(\sigma_3)$ that are on $\lambda_4$, satisfies similar properties. 
One gets a loop $C$ by taking the union of $f(\sigma_1)$, $f(\sigma_3)$, $\sigma_2$ and $\sigma_4$. The vector field $z\mapsto f^{-1}(z)-z$ does not vanish on $C$, let us explain why its index on $C$ is equal to $1$ or $-1$, which implies that there exists at least one fixed point of $f$ in the bounded component  of $\R^2\setminus C$. The set of orientation preserving homeomorphisms $h$ of $\R^2$ being path connected for the compact open topology, the value of the index of the vector field $z\mapsto h\circ f^{-1}\circ h^{-1}(z)-z$ on $h(C)$ does not depend on the choice of $h$.  Applying Homma's theorem, one can find an orientation preserving homeomorphism $h$ such that :
$$\begin{aligned} h(\lambda_2) &=\{0\}\times\R, \\ h(\lambda_4)&=\{1\}\times \R,\\ h(f(\sigma_1)) &=[0,1]\times\{0\}, \\ h(f(\sigma_3)) &=[0,1]\times\{1\},\end{aligned} $$
or 
$$ \begin{aligned} h(\lambda_2) &=\{0\}\times\R, \\ h(\lambda_4)&=\{1\}\times \R, \\ h(f(\sigma_3)) &=[0,1]\times\{0\}, \\h(f(\sigma_1)) &=[0,1]\times\{1\}.\end{aligned}$$
There is no loss of generality by supposing that the first situation occurs. The family $(\lambda_i)_{1\leq i\leq 4}$ being cyclically ordered, there are two possibilities: in the first case, $\overline E_1$ is contained in $(0,1)\times(-\infty, 0)$ and $\overline E_3$ in $(0,1)\times (1,+\infty)$; in the second case, $\overline E_1$ is contained in $(0,1)\times (0+\infty)$ and $\overline E_3$ in $(0,1)\times (-\infty, 1)$. It is very easy to compute the index of  the vector field $z\mapsto h\circ f^{-1}\circ h^{-1}(z)-z$ on the square $h(C)$: in the first case,  the vector field is pointing on the right on $\{0\}\times[0,1]$,  pointing downwards on $[0,1]\times\{0\}$, pointing on the left on  $\{1\} \times[0,1]$, pointing upwards on $[0,1]\times\{1\}$ and the index is $-1$;  in the second case,  the vector field is pointing on the right on $\{0\}\times[0,1]$,  pointing upwards on $[0,1]\times\{0\}$, pointing on the left on  $\{1\} \times[0,1]$, pointing downwards on $[0,1]\times\{1\}$ and the index is $+1$.\end{proof}

\begin{remark*} A special case where this criterion can be applied is when every line $\lambda_i$ is free (meaning that $f(\lambda_i)\cap \lambda_i=\emptyset$) and  $f(\lambda_i)\cap \lambda_j\not=\emptyset$  if $i\in\{1,3\}$ and $j\in \{2,4\}$. The proposition above was proved in \cite{LT2} when every $\lambda_i$ is a Brouwer line and $f(\lambda_i)\cap \lambda_j\not=\emptyset$  if $i\in\{1,3\}$ and $j\in \{2,4\}$. The proof above tells us that the arguments given in \cite{LT2} are still valid with some weaker hypothesis. 
\end{remark*}

\subsection{Some forcing results}\label{s.forcing} 
In this sub-section, we will state another result which will be essential in the proofs of Propositions  \ref{prop:twist}  and  \ref{prop:connexions}. It is naturally related to the forcing lemma for Brouwer lines that is stated in \cite{LT1} but will concern free lines instead of Brouwer lines and will ``include'' a dynamics among these lines.

Denote
 $$p_1:\tilde\A\to\R\enskip\mathrm{and}\enskip p_2:\tilde\A\to[0,1]$$
 the horizontal and vertical projections defined on $\tilde\A=\R\times[0,1]$. 
 
 Write $\mathrm{Homeo}_*(\tilde\A)$ for the set of orientation preserving homeomorphisms  of $\tilde\A$ that lets invariant each boundary line. The boundary lines can be naturally ordered, transposing by $p_1$ the usual order of the real line.

Denote $\tilde{\mathcal E}_{0}$ the set of lines $\tilde\lambda_0:\R\to \R \times(0,1)$ such that there exist $(\tilde a_0^-,0)\not=(\tilde a_0^+,0)$  in $\R\times\{0\}$ satisfying 
$$\lim_{t\to-\infty}\tilde  \lambda_0(t)=(\tilde a_0^-,0),\enskip \lim_{t\to+\infty}\tilde  \lambda_0(t)=(\tilde a_0^+,0).$$ We will say that $(\tilde a_0^-,0)$ and $(\tilde a_0^+,0)$ are the {\it ends }of $\tilde\lambda$.\footnote{ We could have defined  $\tilde{\mathcal E}_{0}$ to be the set of segments of $\tilde\A$ whose ends are on $\R\times\{0\}$ and whose interior is in the interior of $\tilde\A$ but the object that will appear naturally for the applications are lines and not segments.}
 We can define a relation $\prec$ on $\tilde{\mathcal E}_{0}$, writing $\tilde\lambda_0\prec\tilde\lambda'_0$ if:
\begin{itemize} 
\item
$\tilde\lambda_0\cap \tilde\lambda'_0=\emptyset$;
\item the smallest end of $\tilde\lambda_0$ is smaller than the smallest end of $\tilde\lambda'_0$ and the highest end of $\tilde\lambda_0$ is smaller than the highest end of $\tilde\lambda'_0$.

\end{itemize}

Note that if $\tilde\lambda_0\prec\tilde\lambda'_0$, then the highest end of $\tilde\lambda_0$ is not higher  than the smallest end of $\tilde\lambda'_0$ because $\tilde\lambda_0\cap \tilde\lambda'_0=\emptyset$ but it can be equal. The relation $\prec$ is not transitive owing to the first condition. Nevertheless, if  $\tilde{\mathcal F}_{0}$ is a subset of $\tilde{\mathcal E}_{0}$ such that the lines $\widetilde\lambda_0\in \tilde{\mathcal F}_{0}$ are pairwise disjoint, then the restriction of $\prec$ to  $\tilde{\mathcal F}_{0}$ induces an order $\preceq$ (not necessarily total) defined as follows:
$$\tilde \lambda_0\preceq\tilde  \lambda'_0\Leftrightarrow \tilde\lambda_0=\tilde\lambda'_0\enskip \mathrm{or}\enskip \tilde\lambda_0\prec\tilde\lambda'_0.$$ 
Every $\tilde f\in \mathrm{Homeo}_*(\tilde\A)$ naturally acts on $\tilde{\mathcal E}_{0}$ and it holds that
 $$ \tilde \lambda_0\prec \tilde \lambda'_0\Longrightarrow \tilde f(\tilde \lambda_0)\prec\tilde f( \tilde \lambda'_0).$$

 Similarly, one can define the set $\tilde{\mathcal E}_{1}$ of lines
 $\tilde\lambda_1:\R\to \R \times(0,1)$ such that there exist $(\tilde a_1^-,1)\not=(\tilde a_1^+,1)$  in $\R\times\{1\}$, the ends of  $\tilde\lambda_1$,   satisfying
$$\lim_{t\to-\infty}\tilde  \lambda_1(t)=(\tilde a_1^-,1),\enskip \lim_{t\to+\infty}\tilde  \lambda_1(t)=(\tilde a_1^+,1).$$
Moreover, one can define a relation $\prec$ on  $\tilde{\mathcal E}_{1}$ like we did on $\tilde{\mathcal E}_{0}$.

\medskip

We now fix $\tilde f\in \mathrm{Homeo}_*(\tilde\A)$ until the end of the section.  Consider $\tilde \lambda_0\in\tilde  {\mathcal E}_0$, $\tilde \lambda_1\in\tilde  {\mathcal E}_1$ and $n\geq 1$.  We will write $\tilde\lambda_0\overset{n} {\longrightarrow}\tilde \lambda_1$ in the case where $\tilde f^n( \tilde \lambda_0)\cap \tilde \lambda_1\not=\emptyset$. The following result is immediate:

\begin{lemma} \label{lemma:properties} Suppose that  $\tilde \lambda_0\in \tilde {\mathcal E}_{0}$,  $\tilde \lambda_1\in \tilde {\mathcal E}_{1}$ and $n\geq 1$ 
satisfy $\tilde \lambda_0\overset{n} {\longrightarrow} \tilde \lambda_1$. Then it holds that
$$\enskip \tilde \lambda_0\overset{n+1} {\longrightarrow}  \tilde f(\tilde \lambda_1), \enskip \tilde f^{-1}(\tilde \lambda_0)\overset{n+1} {\longrightarrow}   \tilde \lambda_1.$$
Moreover, if $\tilde h\in \mathrm{Homeo}_*(\tilde\A)$ commutes with $\tilde f$, then 
$$\tilde h (\tilde \lambda_0)\overset{n} {\longrightarrow}   \tilde h(\tilde \lambda_1). $$

\end{lemma}

The next result is less obvious:

\begin{lemma} \label{lemma:forcing} Suppose that $\tilde \lambda_0\in \tilde{\mathcal E}_{0}$,  $\tilde \lambda'_0\in \tilde{\mathcal E}_{0}$ ,
$\tilde \lambda_1\in \tilde{\mathcal E}_{1}$,  $\tilde \lambda'_1\in \tilde{\mathcal E}_{1}$, $n\geq 1$,  $n'\geq 1$ satisfy
$$\tilde \lambda_0\overset{n} {\longrightarrow} \tilde \lambda_1, \enskip \tilde \lambda'_0\overset{n'} {\longrightarrow} \tilde \lambda'_1$$
and 
$$ \tilde f^n(\tilde \lambda_0)\prec \tilde\lambda'_0, \enskip \tilde\lambda'_1\prec f^{n'}(\tilde\lambda_1), \enskip \tilde\lambda'_0\cap \tilde\lambda_1=\emptyset.$$ Then, one has
$$\tilde \lambda_0\overset{n+n'} {\longrightarrow} \tilde \lambda'_1.$$

\end{lemma}

\begin{proof} The lines $\tilde f^n(\tilde\lambda_0)$ and $\tilde\lambda_1$ intersect, so $\tilde f^n(\tilde\lambda_0)\cup \tilde\lambda_1$ is connected. Similarly, $\tilde\lambda'_0$  and $\tilde f^{-n'}(\tilde\lambda'_1)$ intersect, so $\tilde\lambda'_0\cup\tilde f^{-n'}(\tilde\lambda'_1)$ is connected.
By hypothesis, one has $\tilde f^{-n'}(\tilde\lambda'_1)\prec \tilde\lambda_1$. The fact that 
$ \tilde f^n(\tilde \lambda_0)\prec \tilde\lambda'_0$ and $ \tilde f^{-n'}(\tilde\lambda'_1)\prec \tilde\lambda_1$ implies that
  $\tilde f^n(\tilde\lambda_0)\cup \tilde\lambda_1$ and $\tilde\lambda'_0\cup\tilde f^{-n'}(\tilde\lambda'_1)$ intersect.
Noting that all sets
$$\tilde\lambda'_0\cap \tilde\lambda_1, \enskip  \tilde f^n(\tilde \lambda_0)\cap \tilde\lambda'_0,\enskip  \tilde f^{-n'}(\tilde\lambda'_1)\cap \tilde\lambda_1$$ are empty, we deduce that $ \tilde f^n(\tilde \lambda_0)$ and $\tilde f^{-n'}(\tilde\lambda'_1)$ intersect, which means that $\tilde \lambda_0\overset{n+n'} {\longrightarrow} \tilde \lambda'_1.$\end{proof}

\begin{remark*}  Lemma \ref{lemma:forcing} is a slight modification of the forcing lemma given in  \cite{LT1}. Of course, its conclusion still holds under the assumptions
$$\tilde\lambda'_0\prec  \tilde f^n(\tilde \lambda_0), \enskip \tilde f^{n'}(\tilde\lambda_1)\prec \tilde\lambda'_1, \enskip \tilde\lambda'_0\cap \tilde\lambda_1=\emptyset.$$
\end{remark*}

Let us state now the main result, supposing that $\tilde f$ commutes with $T:(x,y)\mapsto(x+1,y)$ or equivalently supposing that it lifts a homeomorphism of $\A=\T\times[0,1]$.

\begin{proposition} \label{prop:connexions-general} Let $\rho_0$ and $\rho_1$ be the rotation numbers induced by $\tilde f$ on $\R\times\{0\}$ and $\R\times\{1\}$ respectively. Suppose that $\tilde \lambda_0\in \tilde {\mathcal E}_{0}$,  $\tilde \lambda_1\in \tilde {\mathcal E}_{1}$ satisfy the following:

\begin{enumerate}

\item for every $n>0$ and every $p\in\Z$, one has $\tilde f^n(\tilde\lambda_0)\cap T^p(\tilde \lambda_0)=\emptyset,$

\item for every $n>0$ and every $p\in\Z$,  one has $\tilde f^n(\tilde\lambda_1)\cap T^p(\tilde \lambda_1)=\emptyset,$

\item for every $n>0$ and every $p\in\Z$,  one has $\tilde f^n(\tilde\lambda_1)\cap T^p(\tilde \lambda_0)=\emptyset,$

\item there exists $n_0\geq 1$ such that  $\tilde \lambda_0\overset{n_0} {\longrightarrow} \tilde \lambda_1$.

\end{enumerate}
Then, if $p\in\Z$ and $q\geq 1$ satisfy
$$\rho_0(n_0+q)<p<\rho_1(n_0+q),$$ or
$$\rho_1(n_0+q)<p<\rho_0(n_0+q),$$  it holds that $$\tilde\lambda_0\overset{2n_0+q} {\longrightarrow}   T^p\tilde\lambda_1.$$

\end{proposition}

\begin{proof}

We can suppose that $\rho_0(n_0+q)<p<\rho_1(n_0+q)$, the other case being similar. Let $(\tilde z_0^-,0)$ and $(\tilde z_0^+,0)$ be the ends of $\tilde\lambda_0$. One can suppose for instance that $\tilde z_0^-<\tilde z_0^+$. From the relation $\rho_0(n_0+q)<p$, we obtain that 
$$\tilde f^{n_0+q}(\tilde z_0^-,0)< T^p(\tilde z_0^-,0), \enskip \tilde f^{n_0+q}(\tilde z_0^+,0)< T^p(\tilde z_0^+,0).$$
As a consequence of (1) we deduce that $$\tilde f^{n_0+q}(\tilde\lambda_0)\prec T^p(\tilde\lambda_0).$$

Similarly, as a consequence of (2) and of the relation $p<\rho_1(n_0+q)$, we deduce that $$T^p(\tilde\lambda_1) \prec \tilde f^{n_0+q}(\tilde\lambda_1).$$
Writing  $\tilde f^{n_0+q}(\tilde\lambda_1)=\tilde f^{n_0}(\tilde f^{q}(\tilde\lambda_1))$, using (3) and the relations
$$\tilde\lambda_0\overset{n_0+q} {\longrightarrow} \tilde f^q(\tilde\lambda_1),  \enskip T^p\tilde\lambda_0\overset{n_0} {\longrightarrow}  T^p\tilde\lambda_1,$$
we conclude by Lemma  \ref{lemma:forcing}, that 
$$\tilde\lambda_0\overset{2n_0+q} {\longrightarrow} T^p\tilde\lambda_1.$$\end{proof}

\begin{remark*} Of course, what has been done in this section can be extended to any abstract annulus (meaning every topological space homeomorphic to $\A$) and its universal covering space, the sets $\tilde {\mathcal E}_{0}$ and $\tilde {\mathcal E}_{1}$ being defined relative to the two boundary lines of the universal covering space.
\end{remark*}

\section{Dehn twist maps}\label{s.dehn}

The goal of this section is to prove Proposition  \ref{prop:twist}, which means to prove that  if $S$ is an orientable closed surface of genus $g\geq2$ and $f$ a non wandering homeomorphism of $S$  isotopic to a Dehn twist map, then $f$ has periodic points of period arbitrarily large.

We will fix from now on a Dehn  twist map $h$ on $S$ and a homeomorphism $f$ isotopic to $h$ (note that $f$ is orientation preserving). We will begin by stating some results that can be found in \cite{LS}. We denote $(A_{i})_{i\in I}$ the family of twisted annuli and  $(r_i)_{i\in I}$  the family of twist coefficients. Fix an annulus $A=A_{i_0}$ and then choose a connected component $\tilde A$ of $\tilde\pi^{-1}(A)$, where $\tilde\pi: \tilde S\to S$ is the universal covering projection. The boundary of $\tilde A$ is the union of two lines $\tilde \lambda_1$ and $\tilde \lambda_2$, each of them lifting a boundary circle of $A$, denoted respectively $\lambda_1$ and $\lambda_2$. We orient $\tilde\lambda_1$ and $\tilde\lambda_2$ in such a way that $\tilde\lambda_2\subset L(\tilde\lambda_1)$ and $\lambda_1\subset R(\tilde\lambda_2)$. There exists $T_0\in G$, uniquely defined, such that  $\tilde\lambda_1$ and $\tilde\lambda_2$ are $T_0$-lines. Note that the  stabilizer of $\tilde A$ in $G$ is the infinite cyclic group generated by  $T_0$. There exists a lift $\tilde h$ of $h$, uniquely defined, that fixes every point of $\tilde \lambda_1$. This lift coincides with $T_0^{r_{i_0}}$ or $T_0^{-r_{i_0}}$ on $\tilde \lambda_2$. Replacing $r_{i_0}$ with $-r_{i_0}$ if necessary, we can suppose that we are in the first case. The map $\tilde h$ fixes every  point of the unique connected component of $\tilde \pi^{-1}(S\setminus \cup_{i\in I}A_i)$ whose closure contains $\tilde\lambda_1$. This component is included in $R(\tilde\lambda_1)$ and its closure in $\overline{\tilde S}$ meets a (unique) component $\tilde J_1$ of $\partial\tilde S\setminus\{\alpha(T_0), \omega(T_0)\}$ because no connected component of $S\setminus \cup_{i\in I}A_i$ is an annulus. Consequently, the extension of $\tilde h$ to $\overline{\tilde S}$, still denoted $\tilde h$, admits fixed points in $\tilde J_1$. For the same reason $\tilde h\circ T_0^{-r_{i_0}}$ admits fixed points in the other component, denoted $\tilde J_2$.

Note that $\tilde h$ commutes with $T_0$ and so lifts a homeomorphism $\hat h$ of the open annulus $\hat S=\tilde S/T_0$ that preserves the orientation and fixes the two ends of $\hat S$. The map $\hat h$ can be extended to the compact annulus $\overline{\hat S}=\left(\overline {\tilde S}\setminus \{\alpha(T_0), \omega(T_0)\}\right)/T_0 $. It contains fixed points on the added circles $\hat J_1= \tilde J_1/T_0$ and $\hat J_2= \tilde J_2/T_0$, the fixed points in $\hat J_1$ being lifted to fixed points of $\tilde h$, the fixed points in $\hat J_2$ being lifted to fixed points of $\tilde h\circ T_0^{-r_{i_0}}$. 

The map $f$ being isotopic to $h$ admits a unique lift $\tilde f$ such that $[\tilde f]=[\tilde h]$ and this lift can be extended to a homeomorphism of $\overline{\tilde S}$ that coincides with $\tilde h$ on $\partial \tilde S$. The map  $\tilde f$ commutes with $T_0$ and lifts a homeomorphism $\hat f$ of $\hat S$. The map $\hat f$ can be extended to a homeomorphism of $\overline{\hat S}$ that coincides with $\hat h$ on $\hat J_1$ and $\hat J_2$. Consequently, $\hat f$ admits fixed points on the boundary circles of $\overline{\hat S}$, the ones on $\widehat J_1$ having a rotation number equal to zero for the lift $\tilde f$, the ones on $\widehat J_2$ having a rotation number equal to $r_{i_0}$ for the lift $\tilde f$: the map $\hat f$ satisfies a boundary twist condition.

\begin{proposition} \label{prop:essentialannular} At least one of the following situations holds:

\begin{enumerate}
\item the map $f$ has periodic points of period arbitrarily large;

\item there exists an essential simple loop $\hat \lambda\subset \hat S$ such that $\hat f(\hat \lambda)\cap\hat \lambda=\emptyset$.
\end{enumerate}

\end{proposition}

This result was proved in  \cite{LS} in a weaker form, asking for infinitely many periodic points instead of periodic points of period arbitrarily large. The proof below is a slight modification of the proof in \cite{LS}.

\begin{proof}  By Theorem \ref{th:PB}, one knows that if (2) is not true, then for every $p/q\in[0,r_0]$, written in an irreducible way, there exists a periodic point $\hat z$ of period $q$ and rotation number $p/q$ for the lift $\tilde f$. One can easily prove (see \cite{LS}) that for every non trivial compact interval $J\subset(0,r_0)$, there exists a compact set $K\subset \hat S$ such that every periodic orbit of rotation number $p/q\in J$ meets $K$. Let $(p_m/q_m)_{m\geq 0}$ be a sequence in $J$, such that the sequence $(q_m)_{m\geq 0}$  is increasing and consists of prime numbers. For every $m\geq 0$, choose a periodic point $\hat z_m\in K$ of period $q_m$ and rotation number $p_m/q_m$. Taking a subsequence if necessary, one can suppose that the sequence $ (\hat z_m)_{m\geq 0}$ converges to a point $\hat z$. Moreover, $\hat z_m$ projects onto a point $z_m\in S$ satisfying $f^{q_m}(z_m)=z_m$. The integer $q_m$ being prime, $z_m$ has period $q_m$ if it is not fixed. To prove the proposition it remains to show that $z_m$ is not fixed if $m$ is large enough. If it is not the case, taking a sub-sequence if necessary, one can suppose that all $z_m$ are fixed. The sequence $ (z_m)_{m\geq 0}$ converges to  $z$, the projection of $\hat z$, and this point is a fixed point. Take a lift $\tilde z\in \tilde S$ of $\hat z$. There exists a sequence $(T_n)_{n\geq 1}$ in $G$ such that $\tilde f^n(\tilde z)= T_n(\tilde z)$. But if $m$ is large enough and $\tilde z_m$ is the lift of $\hat z_m$ that is close to $\tilde z$, then one has $\tilde f^n(\tilde z_m)= T_n(\tilde z_m)$, for every $n\geq  1$. This is impossible because the integers $q_m$ are all distinct. 
\end{proof}

\bigskip
By Proposition  \ref{prop:essentialannular}, it is sufficient to prove the following result to get Proposition  \ref{prop:twist}:

\begin{proposition} \label{prop:twisthorseshoe} If $f$ is non wandering and if there exists an essential simple loop $\hat \lambda$ such that $\hat f(\hat \lambda)\cap\hat \lambda=\emptyset$, then $f$ has periodic points of period arbitrarily large. 
\end{proposition}

Let us begin by proving:

\begin{lemma} \label{lemma:intprop} Suppose that $f$ is non wandering and that there exists an essential simple loop $\hat \lambda$ such that $\hat f(\hat \lambda)\cap\hat \lambda=\emptyset$. Then:
\begin{enumerate}
\item the annulus $A_{i_0}$ does not separate $S$ (its complement is connected);

\item the loop $\hat \lambda$ projects onto a simple loop $\check \lambda\subset \check S$ such that $\check f(\check \lambda)\cap\check\lambda=\emptyset$, where $\check S$ is the cyclic cover of $S$ naturally associated to $A_{i_0}$ and $\check f$ the homeomorphism of $\check S$ lifted by $\hat f$.\end{enumerate}

\end{lemma}

The surface $\check S$ is the normal covering space of $S$, whose group of automorphisms is infinite cyclic, and such that the preimage of $A_{i_0}$ by the covering projection is the union of disjoint separating annuli homeomorphic to $A_{i_0}$.  The result is proved in \cite{LS} assuming $f$ lets invariant a totally supported finite measure. We just need to verify that the arguments given in \cite{LS} are still valid in this wider situation.

\begin{proof} 
 The proof given in  \cite{LS} is based on the fact that if the conclusions of Lemma \ref{lemma:intprop}  are not satisfied, there exists a non empty wandering open set $\hat U$ of $\hat f$ such that $\bigcup_{n\geq 0}\hat f^n(\hat U)$ or $\bigcup_{n\geq 0}\hat f^{-n}(\hat U)$ is relatively compact. So, by Proposition \ref{prop:wandering}, the proof extends to the case of a non wandering homeomorphism.  Let us explain briefly the arguments.

The loop $\hat \lambda$ can be lifted to a $T_0$-line $\tilde\lambda_0$ of $\tilde S$. The orientation of $\tilde\lambda_0$ induces an orientation on $\hat \lambda$. The loops $\hat f(\hat\lambda)$  and  $\hat f^{-1}(\hat\lambda)$ belong to different components of $\hat S\setminus\hat  \lambda$. Replacing $f$ with $f^{-1}$ if necessary, one can suppose that $\hat f(\hat\lambda)$ is on the left of $\hat\lambda$ and $\hat f^{-1}(\hat\lambda)$ on its right. This implies that $\tilde \lambda_0$ is a Brouwer line of $\tilde f$. 
The lift $\tilde f$ acts on the set of lifts of $\hat\lambda$ in a natural way. Indeed if $\tilde\lambda	=T\tilde \lambda_0$, $T\in G$,  is another lift, then $\tilde\lambda$ is a $TT_0T^{-1}$-line and one can define 
$$[\tilde f] (\tilde \lambda)= [\tilde f](T)(\tilde\lambda_0) = \tilde f\circ T\circ \tilde f^{-1}(\tilde\lambda_0).$$The line $\tilde\lambda_0$ being a Brouwer line of $\tilde f$, it holds that:
  $$\tilde f\left(\overline{L(\tilde\lambda)}\right)\subset L([\tilde f](\tilde\lambda)), \enskip \tilde f^{-1}\left(\overline{R(\tilde\lambda)}\right)\subset R([\tilde f]^{-1}(\tilde\lambda)).$$

Define $\Lambda^-$ as being the set of  lifts $\tilde \lambda =T(\tilde\lambda_0)$ such that:
\begin{itemize} 
\item$\alpha(TT_0T^{-1})$ and $\omega(TT_0T^{-1})$ are on the right of $\tilde \lambda_0$;
\item$\alpha(T_0)$ and $\omega(T_0)$  are on the right of $\tilde\lambda$.
\end{itemize}
Similarly, define $\Lambda^+$ as being the set of  lifts $\tilde \lambda =T(\tilde\lambda_0)$  such that:
\begin{itemize} 
\item$\alpha(TT_0T^{-1})$ and $\omega(TT_0T^{-1})$ are on the left of $\tilde \lambda_0$;
\item$\alpha(T_0)$ and $\omega(T_0)$  are on the left of $\tilde \lambda$.
\end{itemize}

It is not very difficult to see that the conclusions of Lemma \ref {lemma:intprop} will occur if we prove that $\tilde\lambda\cap \tilde \lambda_0=\emptyset$ for every $\tilde\lambda\in \Lambda_-\cup \Lambda_+$.  Fix $\tilde\lambda\in \Lambda_{-}$. It projects onto a line of $\hat S$, denoted $\hat{\tilde\lambda}$, that joins the end on the right of $\hat \lambda$ to itself. It separates $\hat S$ into two components, $R(\hat{\tilde\lambda})$ on its right  and $L(\hat{\tilde\lambda})$ on its left.  
Of course, one has $\widehat{T_0(\tilde\lambda)}= \hat{\tilde\lambda}$. There are finitely many lines $\hat{\tilde\lambda}$, $\tilde\lambda\in \Lambda_-$, that meet $\hat \lambda$, we want to prove that none of them does it. The set
 $$\hat K=\overline {L(\hat \lambda)}\cap \left(\bigcup_{\tilde\lambda\in \Lambda_-} \overline{L(\hat{\tilde\lambda})}\right)$$ 
 is compact and satisfies
$\hat f(\hat K) \subset \mathrm{int} (\hat K)$. Indeed, if $\hat z\in  \overline {L(\hat \lambda)}\cap\overline{L(\hat{\tilde\lambda})}$, then $\hat f(\hat z)\in  L(\hat \lambda)\cap L(\widehat{[f](\tilde\lambda)})$. 
In case, there exists $\tilde\lambda\in \Lambda_-$ such that $\tilde\lambda\cap \tilde \lambda_0\not=\emptyset$, we deduce that $\hat U=\mathrm{int} (\hat K)\setminus \hat f(\hat K) $ is a non empty wandering open set of $\hat f$ such that $\bigcup_{n\geq 0}\hat f^n(\hat U)$ is contained in $\hat K$. 
 We can apply Proposition \ref{prop:wandering} to get a contradiction. Similarly, we can construct a non empty wandering open set $\hat U$ of $\hat f$ such that $\bigcup_{n\geq 0}\hat f^{-n}(\hat U)$ is relatively compact in case there exists $\tilde\lambda\in \Lambda_+$ such that $\tilde\lambda\cap \tilde \lambda_0\not=\emptyset$.\end{proof}

We will suppose until the end of the section that $f$ satisfies the hypothesis of Proposition \ref{prop:twisthorseshoe} and Lemma \ref{lemma:intprop}: $f$ is non wandering and there exists an essential simple loop $\hat \lambda$ such that $\hat f(\hat \lambda)\cap\hat \lambda=\emptyset$. More precisely we will suppose that $\hat f(\hat\lambda)$ is on the left of $\hat\lambda$ and $\hat f^{-1}(\hat\lambda)$ on its right, meaning that the lift $\tilde \lambda_0$ of $\hat\lambda$ that is a $T_0$-line, is a Brouwer line of $\tilde f$.

We begin by using Lemma \ref{lemma:intprop}. Let $\check T$ be the generator of the group of covering automorphisms of the covering map $\check \pi: \check S\to S$ such that $\check T(\check \lambda_1)$ is on the left of $\check\lambda_1$, where $\check \lambda_1$ is a lift of $\lambda_1$ (one of the boundary curves of $A_{i_0}$). It is possible that $\check T(\check \lambda)\cap \check\lambda\not=\emptyset$ but if  $s\geq 1$ is large enough, then $\check T^s(\check\lambda)$ does not meet $\check\lambda$ and is on the left of $\check \lambda$. Replacing $S$ with $\check S/\check T^s$ and $f$ with the homeomorphism of $\check S/\check T^s$ that is lifted by $\check f$, we can always suppose that $s=1$, meaning that $\hat\lambda$ projects onto a simple loop $\lambda$ of $S$ (indeed, by Proposition  \ref{prop:NWH} the new map will be non wandering and if it has periodic  points of period arbitrarily high, so will have $f$).  

$\hat {\tilde \lambda}$ Let us explain now how we will apply the results of Sub-section \ref{s.forcing} (forcing theory). Considering the annulus $\overline{\hat S}$ and its universal covering space $\overline{\tilde S}\setminus\{\alpha(T_0), \omega(T_0)\}$, we can define the  sets $\tilde{\mathcal E}_1$ and $\tilde{\mathcal E}_2$. Recall that there is a well defined relation $\prec$ on these sets. Every lift $\tilde\lambda\not=\tilde\lambda_0$ of $\lambda$ is disjoint from $\tilde\lambda_0$. If it is on the right of $\tilde\lambda_0$ its two ``ends''  are on $\tilde J_1$, and so $\tilde\lambda$ belongs to $\tilde{\mathcal E}_1$, if it is on the left of $\tilde\lambda_0$ the ends are on $\tilde J_2$, and so $\tilde\lambda$ belongs to $\tilde{\mathcal E}_2$.  For every integers $m<n$, denote $\check S_{[m,n]}$ the compact surface bordered by $\check T^{m}(\check \lambda)$ and $\check T^{n}(\check \lambda)$. Then note respectively $\tilde S_{[-1,0]}$ and $\tilde S_{[0,1]}$ the connected component of the preimage of $\check S_{[-1,0]}$ and $\check S_{[0,1]}$ by the universal covering projection that contains $\tilde \lambda_0$ in its boundary. 
The boundary of $\tilde S_{[-1,0]}$ is a disjoint union of lifts of $\lambda$, the ones that bound $\tilde S_{[-1,0]}$ on their right side being lifts of $\check \lambda$, the ones that bound $\tilde S_{[-1,0]}$ on their left side being lifts of $\check T^{-1}(\check \lambda)$. Denote $\tilde{\mathcal L}_{1}$ the set of lifts of $\lambda$ different from $\tilde\lambda_0$, that are on the boundary of $\tilde S_{[-1,0]}$. Denote $\tilde{\mathcal L}_{1,l}\subset\tilde{\mathcal L}_{1}$ the subset of lines in   $\tilde{\mathcal L}_{1}$ that lift $\check T^{-1}(\check\lambda)$ and $\tilde{\mathcal L}_{1,r}\subset \tilde{\mathcal L}_{1}$ the subset of lines in  $\tilde{\mathcal L}_{1}$ that lift $\check\lambda$. Note also that the relation $\prec$ defines a total order on   $\tilde{\mathcal L}_{1}$, setting
$$\tilde\lambda_1\preceq \tilde \lambda'_1\Leftrightarrow \tilde\lambda_1\prec \tilde \lambda'_1\enskip\mathrm{or}\enskip \tilde\lambda_1= \tilde \lambda'_1.$$ Similarly, denote $\tilde{\mathcal L}_{2}$ the set of lifts of $\lambda$ different from $\tilde\lambda_0$ that are on the boundary of $\tilde S_{[0,1]}$, denote $\tilde{\mathcal L}_{2,l}\subset \tilde{\mathcal L}_{2}$ the subset of lines in $\tilde{\mathcal L}_{2}$  that lift $\check\lambda$ and $\tilde{\mathcal L}_{2,r}\subset \tilde{\mathcal L}_{2}$ the subset of lines in  $\tilde{\mathcal L}_{2}$ that lift $\check T(\check\lambda$). Here again,  $\prec$ induces a total order on  $\tilde{\mathcal L}_{2}$. There is a natural action of $\Z^2$ on these sets. Note first that $\tilde{\mathcal L}_1$,  $\tilde{\mathcal L}_{1,r}$, $\tilde{\mathcal L}_{1,l}$, $\tilde{\mathcal L}_2$, $\tilde{\mathcal L}_{2,r}$, $\tilde{\mathcal L}_{2,l}$, are invariant by $T_0$ because $\tilde \lambda_0$ is invariant by $T_0$. These sets are also invariant by the map $[\tilde f]$ defined on the set of lifts of $\lambda$. Moreover, $T_0$ and $[\tilde f]$ commute on our sets.

One deduces that: 
$$\begin{aligned}\tilde\lambda_1\in \tilde{\mathcal L}_{1,r},\enskip \tilde\lambda\in \tilde{\mathcal L}_1\cup \tilde{\mathcal L}_2 , \enskip k>0&\Rightarrow \tilde f^k(\tilde\lambda_1)\cap \tilde\lambda=\emptyset,\\
\tilde\lambda_1\in \tilde{\mathcal L}_{1,l},\enskip \tilde\lambda\in \tilde{\mathcal L}_1\cup \tilde{\mathcal L}_2 ,\enskip k>0&\Rightarrow \tilde f^{-k}(\tilde\lambda_1)\cap \tilde\lambda=\emptyset,\\
\tilde\lambda_2\in \tilde{\mathcal L}_{2,r},\enskip \tilde\lambda\in \tilde{\mathcal L}_1\cup \tilde{\mathcal L}_2 ,\enskip k>0&\Rightarrow \tilde f^k(\tilde\lambda_2)\cap \tilde\lambda=\emptyset,\\
\tilde\lambda_2\in \tilde{\mathcal L}_{2,l},\enskip \tilde\lambda\in\tilde{\mathcal L}_1\cup \tilde{\mathcal L}_2 ,\enskip k>0&\Rightarrow \tilde f^{-k}(\tilde\lambda_2)\cap \tilde\lambda=\emptyset.
\end{aligned}$$Consequently, it holds that
$$\begin{aligned}\tilde\lambda_1\in \tilde{\mathcal L}_{1,r},\enskip \tilde\lambda'_1\in \tilde{\mathcal L}_{1,r},\enskip k\in\Z\setminus\{0\}&\Rightarrow \tilde f^k(\tilde\lambda_1)\cap \tilde\lambda'_1=\emptyset,\\\tilde\lambda_1\in \tilde{\mathcal L}_{1,l},\enskip \tilde\lambda'_1\in \tilde{\mathcal L}_{1,l},\enskip k\in\Z\setminus\{0\}&\Rightarrow \tilde f^k(\tilde\lambda_1)\cap \tilde\lambda'_1=\emptyset,\\
\tilde\lambda_2\in \tilde{\mathcal L}_{2,r},\enskip \tilde\lambda'_2\in \tilde{\mathcal L}_{2,r},\enskip k\in\Z\setminus\{0\}&\Rightarrow \tilde f^k(\tilde\lambda_2)\cap \tilde\lambda'_2=\emptyset,\\\tilde\lambda_1\in \tilde{\mathcal L}_{2,l},\enskip \tilde\lambda'_2\in \tilde{\mathcal L}_{2,l},\enskip k\in\Z\setminus\{0\}&\Rightarrow \tilde f^k(\tilde\lambda_2)\cap \tilde\lambda'_2=\emptyset.
\end{aligned}$$

Let us state now the key result implying Proposition  \ref{prop:twisthorseshoe}. We will postpone its proof and begin by explaining how to get Proposition \ref  {prop:twisthorseshoe} from it.

\begin{proposition} \label{proposition:connections-better} There exist $\tilde \lambda_1\in \tilde{\mathcal L}_{1,l}$, $\tilde \lambda_2\in \tilde{\mathcal L}_{2,r}$, $n_0\in\N$, $\rho_-\in\R$, $\rho_+\in\R$, such that  $\rho_-<\rho_+$ and such that for every $p\in\Z$ and every $q\geq 1$, it holds that
$$\rho_-(n_0+q)<p<\rho_+(n_0+q)\Rightarrow\tilde f^{2n_0+q}(\tilde \lambda_1)\cap \tilde T_0^p(\tilde\lambda_2)\not=\emptyset.$$

Moreover, at least one of the two following properties holds:

\begin{enumerate} 

\item If $$\rho_-(n_0+q)<p<p'<\rho_+(n_0+q), $$ then there exists a segment $\tilde\sigma_1\subset \tilde\lambda_1$ such that:
\begin{itemize}
\item $\tilde f^{2n_0+q}(\tilde \sigma_1)$ joins $T_0^{p}(\tilde\lambda_2)$ and $T_0^{p'}(\tilde\lambda_2)$;
\item the interior of  $\tilde f^{2n_0+q}(\tilde\sigma_1)$ is included in $R(T_0^{p}(\tilde\lambda_2))\cap R(T_0^{p'}(\tilde\lambda_2))$;
\item  $\tilde f^{2n_0+q}(\tilde\sigma_1)$ is included in $\bigcup_{ k\leq 0}\tilde f^{-k}(L(\tilde\lambda_0))$.

\end{itemize}

\item If $$\rho_-(n_0+q)<p<p'<\rho_+(n_0+q), $$ then there exists a segment 
$\tilde\sigma_2\subset \tilde\lambda_2$ such that:

\begin{itemize}
\item $\tilde f^{-2n_0-q}(\tilde\sigma_2)$ joins $T_0^{-p}(\tilde\lambda_1)$ and $T_0^{-p'}(\tilde\lambda_1)$;
\item the interior of  $\tilde f^{-2n_0-q}(\tilde\sigma_2)$ is included in $L(T_0^{-p}(\tilde\lambda_1))\cap  L(T_0^{-p'}(\tilde\lambda_1))$;
\item $\tilde f^{-2n_0-q}(\tilde\sigma_2)$ is included in $\bigcup_{k\geq 0} \tilde f^k(R(\tilde\lambda_0))$.

\end{itemize}
\end{enumerate}
\end{proposition}

We will suppose from now on that Proposition  \ref{proposition:connections-better} is true. We can deduce the following:
\begin{proposition} \label{proposition:connections-better-better} {\it At least one of the two following statements is true:

\begin{enumerate}
\item

There exists  $\tilde \lambda_1\in \tilde{\mathcal L}_{1,l}$ such that for every $s\geq 2$, there exists $m_s\geq 0$ such that for every $m\geq m_s$, there exists $\tilde \lambda_2\in \tilde{\mathcal L}_{2,r}$ such that for every $0<p<p'\leq s$, there exists a segment $\tilde\sigma_1\subset \tilde\lambda_1$  satisfying:
\begin{itemize}
\item $\tilde f^{m}(\tilde \sigma_1)$ joins $T_0^{p}(\tilde\lambda_2)$ and $T_0^{p'}(\tilde\lambda_2)$;
\item the interior of  $\tilde f^{m}(\tilde\sigma_1)$ is included in $R(T_0^{p}(\tilde\lambda_2))\cap R(T_0^{p'}(\tilde\lambda_2))$;
\item  $\tilde f^{m}(\tilde\sigma_1)$ is included in $L(\tilde\lambda_0)$.
\end{itemize}
\item There exists  $\tilde \lambda_2\in \tilde{\mathcal L}_{2,r}$ such that for every $s\geq 2$ there exists $m_s\geq 0$ such that for every $m\geq m_s$, there exists $\tilde \lambda_1\in \tilde{\mathcal L}_{1,l}$ such that for every $0<p<p'\leq s $, there exists a segment $\tilde\sigma_2\subset \tilde\lambda_2$  satisfying:
\begin{itemize}
\item $\tilde f^{-m}(\tilde \sigma_2)$ joins $T_0^{-p}(\tilde\lambda_1)$ and $T_0^{-p'}(\tilde\lambda_1)$;
\item the interior of  $\tilde f^{-m}(\tilde\sigma_2)$ is included in $L(T_0^{-p}(\tilde\lambda_1))\cap L(T_0^{-p'}(\tilde\lambda_1))$;
\item  $\tilde f^{-m}(\tilde\sigma_2)$ is included in $R(\tilde\lambda_0)$.
\end{itemize}

\end{enumerate}}
\end{proposition}

\begin{proof}

Suppose that the first situation of Proposition \ref{proposition:connections-better} holds. Fix $s\geq 2$ and denote $\mathcal K$ the (finite) set of couples $(p,p')$ such that $0<p<p'\leq s$. There exist $b\in\Z$ and $q\geq 1$ such that, for every $(p,p')\in\mathcal K$ it holds that $$\rho_-(n_0+q)<p+b<p'+b<\rho_+(n_0+q).$$For each couple $\kappa=(p,p')\in \mathcal K$, one can choose a segment $\tilde\sigma_{1,\kappa}$ satisfying the three items of  Proposition \ref{proposition:connections-better} (1). The third item tells us that there exists $a_{\kappa}\geq 0$ such that  $\tilde f^{a+2n_0+q}(\tilde\sigma_{1,\kappa})\subset L(\tilde\lambda_0)$ if $a\geq a_{\kappa}$.  Moreover  $\tilde f^{a+2n_0+q}(\tilde\sigma_{1,\kappa})$ joins  $L([\tilde f]^{a}(T_0^{p+b}(\tilde\lambda_2)))$ and $L([\tilde f]^{a}(T_0^{p'+b}(\tilde\lambda_2)))$ and so contains a segment that joins $[\tilde f]^{a}(T_0^{p+b}(\tilde\lambda_2))$ and $[\tilde f]^{a}(T_0^{p'+b}(\tilde\lambda_2))$ and whose interior is included in $R([\tilde f]^{a}(T_0^{p+b}(\tilde\lambda_2)))$ and $R([\tilde f]^{a}(T_0^{p'+b}(\tilde\lambda_2)))$. This segment can be written $\tilde f^{a+2n_0+q}(\tilde\sigma'_{1,\kappa})$ where $ \tilde\sigma'_{1,\kappa}$ is a segment included in $ \tilde\sigma_{1,\kappa}$.  Set $a_{\mathrm{max}}=\max_{\kappa\in K} a_{\kappa}$ and $m_s= a_{\kappa}+2n_0+q$. For every $m\geq m_s$ and for every $\kappa=(p,p')\in \mathcal{K}$, there exists a segment $ \tilde\sigma'_{1,\kappa}\subset \tilde\lambda_1$ satisfying:
\begin{itemize}
\item $\tilde f^{m}(\tilde\sigma'_{1,\kappa})$ joins $T_0^{p}([\tilde f]^{m-2n_0-q}(T_0^b(\tilde\lambda_2)))$ and $T_0^{p'}([\tilde f]^{m-2n_0-q}(T_0^b(\tilde\lambda_2)))$;
\item the interior of  $\tilde f^{m}(\tilde\sigma'_{1,\kappa})$ is included in $R(T_0^{p}([\tilde f]^{m-2n_0-q}(T_0^b(\tilde\lambda_2))))$ and in $R(T_0^{p'}([\tilde f]^{m-2n_0-q}(T_0^b(\tilde\lambda_2))))$;
\item  $\tilde f^{m}(\tilde\sigma'_{1,\kappa})$ is included in $L(\tilde\lambda_0)$.
\end{itemize}
So (1) is satisfied. If the second situation of Proposition \ref{proposition:connections-better} holds, one prove similarly that (2) is satisfied.

\end{proof}

Finally we get:

\begin{proposition} \label{prop:periodic} There exists a sequence $(T_m)_{m\geq m_5}$ in $G$ such that each $T_m$ sends $\tilde S_{[-1,0]}$ onto $\tilde S_{[0,1]}$ and such that $\tilde f^{m}\circ T_m^{-1}$ has a fixed point.
\end{proposition}

\begin{proof}
Suppose that the first item of Proposition  \ref{proposition:connections-better-better} holds and denote $\tilde \lambda_1$ the line of $\tilde{\mathcal L}_{1,l}$ defined in Proposition  \ref{proposition:connections-better-better}. Fix $m\geq m_5$ and denote $\tilde \lambda_2$ the line of $ \tilde{\mathcal L}_{2,r}$ defined in Proposition  \ref{proposition:connections-better-better}. If $T\in G$ sends $\tilde\lambda_1$ onto $\tilde\lambda_0$, it sends $\tilde S_{[-1,0]}$ onto $\tilde S_{[0,1]}$ and  $T_0(\tilde\lambda_1)$ onto an element of $\tilde{\mathcal{L}}_{2,l}$. Moreover 
$T'\in G$ sends $\tilde\lambda_1$ onto $\tilde\lambda_0$ if and only if there exists $k\in\Z$ such that $T'=T_0^kT$. So, there exists an automorphism $T_m\in G$, uniquely defined,  such that 
$$T_m(\tilde\lambda_1)=\tilde\lambda_0\enskip\mathrm{and} \enskip T_0^3(\tilde\lambda_2)\prec T_m(T_0(\tilde\lambda_1))\prec T_0^4(\tilde\lambda_2),$$
and $T_m$ sends $\tilde S_{[-1,0]}$ onto $\tilde S_{[0,1]}$. Note that $T_m$ sends $\tilde\lambda_0$ onto an element of  $\tilde{\mathcal{L}}_{2,r}$. Consider the map $\tilde g=\tilde f^{m}\circ T_m^{-1}$. Every line $\tilde\lambda_2'\in \tilde{\mathcal{L}}_{2,r}\setminus T_m(\tilde \lambda_0)$ is sent  by $T_m^{-1}$ onto an element of $\tilde{\mathcal{L}}_{1,r}$ and its image by $\tilde g$ belongs to $L([\tilde f]^{m}(T_m^{-1}(\tilde \lambda'_2)))$. In particular, $\tilde g (\tilde \lambda'_2)\cap \tilde\lambda'_2=\emptyset$.    Every line $\tilde\lambda_2'\in \tilde{\mathcal{L}}_{2,l}$ is sent  by $T_m^{-1}$ onto an element of $\tilde{\mathcal{L}}_{1,l}$ and its image by $\tilde f^{-m}$ belongs to $R([\tilde f]^{-m}( \lambda'_2))$. The images of $\tilde \lambda'_2$ by $T_m^{-1}$ and by $\tilde f^{-m}$ being disjoint, it holds that $\tilde g (\tilde \lambda'_2)\cap \tilde\lambda'_2=\emptyset$. In fact the only lines on the boundary of $\tilde S_{[0,1]}$ that can meet their image by $\tilde g$ are $\tilde\lambda_0$ and $T_m(\tilde\lambda_0)$. One can choose $i\in\{2,3\}$ and $j\in\{4,5\}$ such that  $T_0^i(\tilde\lambda_2)\not=T_m(\tilde\lambda_0)$ and $T_0^j(\tilde\lambda_2)\not=T_m(\tilde\lambda_0)$. Applying Proposition  \ref{proposition:connections-better-better} to the pairs $(i,j)$ and $(i-1,j-1)$ and using the fact that $\tilde f$ and $T_0$ commute, we deduce that there exist a segment $\tilde\sigma_1\subset \tilde\lambda_1$  and a segment $\tilde\sigma'_1\subset T_0(\tilde\lambda_1)$ satisfying:
\begin{itemize}
\item $\tilde f^{m}(\tilde \sigma_1)$ and $\tilde f^{m}(\tilde \sigma'_1)$ join $T_0^{i}(\tilde\lambda_2)$ and $T_0^{j}(\tilde\lambda_2)$;
\item the interior of  $\tilde f^{m}(\tilde\sigma_1)$ and $\tilde f^{m}(\tilde \sigma'_1)$ are included in $R(T_0^{i}(\tilde\lambda_2))$ and in $R(T_0^{j}(\tilde\lambda_2))$;
\item  $\tilde f^{m}(\tilde\sigma_1)$ and $\tilde f^{m}(\tilde \sigma'_1)$ are included in $L(\tilde\lambda_0)$.

\end{itemize}

Let us summarize the situation:

\begin{itemize}
\item $T_m(\tilde\sigma_1)$ is a segment of $\tilde\lambda_0$ whose image by $\tilde g$ is disjoint from $\widetilde \lambda_0$ and joins $T_0^{i}(\tilde\lambda_2)$ and $T_0^{j}(\tilde\lambda_2)$;  
\item  $T_m(\tilde\sigma'_1)$ is a segment of $T_mT_0(\tilde\lambda_1)$ whose image by $\tilde g$ is disjoint from $T_mT_0(\tilde\lambda_1)$ and joins $T_0^{i}(\tilde\lambda_2)$ and $T_0^{j}(\tilde\lambda_2)$;
\item  $T_0^{i}(\tilde\lambda_2)$ and $T_0^{j}(\tilde\lambda_2)$ are disjoint from there image by $\tilde g$;
\item the lines $\tilde\lambda_0$, $T_0^{j}(\tilde\lambda_2)$, $T_mT_0(\tilde\lambda_1)$ and   $T_0^{i}(\tilde\lambda_2)$ are cyclically ordered.
\end{itemize}

We can use Proposition \ref{prop:fixedpoint} and deduce that $g$ has a fixed point.

 In the case where the second item of Proposition  \ref{proposition:connections-better-better}  holds, we can prove similarly that for every $m\geq m_5$, there exists $T_m\in G$ that sends $\tilde S_{[-1,0]}$ onto $\tilde S_{[0,1]}$, such that $\tilde f^{-m}\circ T_m$ has a fixed point. Note that the image by $T_m$ of the fixed point of  $\tilde f^{-m}\circ T_m$ is a fixed point of  $\tilde f^{m}\circ T_m^{-1}$.
\end{proof}

It remains to explain why Proposition \ref {prop:periodic} implies Proposition  \ref{prop:twisthorseshoe}:
 
\begin{proof}[Proof of Proposition \ref{prop:twisthorseshoe}]  For every $m\geq m_5$, one can find  $\tilde z_m\in\tilde S$ such that $\tilde f^{m}(\tilde z_m)=T_m(\tilde z_m)$ (just take the image by $ T_m^{-1}$ of a fixed point of $\tilde f^{m}\circ T_m^{-1}$).  It projects onto a point $\check z_m\in\check  S$ such that $\check f^{m}(\check z_m)=\check T(\check z_m)$ which itself projects onto a point $z_m\in S$ such that $f^{m}(z_m)=z_m$. It remains to show that the period of $z_m$ is $m$. Suppose that $f^{q}(z_m) =z_m$, where $q\vert m$. Then there exists $p\in\Z$ such that $\check f^{q}(\check z_m)=\check T^p(\check z_m)$ and so $\check f^{m}(\check z_m)=\check T^{pm/q}(\check z_m)$, which implies that $pm=q$. Of course this implies that $m=q$ and $p=1$. \end{proof}

To conclude the section, it remains to prove Proposition  \ref{proposition:connections-better}. Let us begin with the following result:

\begin{lemma} \label{lemma:connections} There exists $\tilde \lambda_1\in \tilde{\mathcal L}_{1,l}$,  $\tilde \lambda_2\in \tilde{\mathcal L}_{2,r}$ and $n\geq 1$ such that $\tilde f^{n} (\tilde\lambda_1)\cap\tilde\lambda_2\not=\emptyset$. 

\end{lemma}

\begin{proof}  One can find an open disk $\check U\subset\check S_{[-1,0]}$  whose image by $\check f$ is on the left of  $\check \lambda$. It is a wandering disk of $\check f$ that projects  onto an open disk $U$ of $S$. There is a  lift $\tilde U\subset \tilde  S_{[-1,0]}$, uniquely defined up to the action of the iterates of $T_0$, whose image by $\tilde f$ is on the left of  $\tilde \lambda_0$. By Lemma  \ref{l:NWH}, there exists $z\in U$ and positive integers  $n_0$, $n_1$, $n_2$ such that $f^{-n_0}(z)$, $f^{n_1}(z)$ and $f^{n_1+n_2}(z)$ belong to $U$. If $\check z$ is the lift of $z$ that belongs to $\check U$, then $\check f^{-n_0}(\check z)$ is on the right of $\check T^{-1}(\check \lambda)$, $\check f^{n_1}(\check z)$ on the left of $\check \lambda$  and $\check f^{n_1+n_2}(\check z)$ on the left of $\check T(\check \lambda)$. So, if $\tilde z$ is the lift of $\check z$ that belongs to $\tilde U$, there exists $\tilde \lambda_1\in \tilde{\mathcal L}_{1,l}$ and $\tilde \lambda_2\in \tilde{\mathcal L}_{2,r}$ such that $\tilde f^{-n_0}(\tilde z)\in  R(\tilde \lambda_1)$ and $\tilde f^{n_1+n_2}(\tilde z)\in  L(\tilde \lambda_2)$. Setting $n=n_0+n_1+ n_2$, one gets that $\tilde f^n( R(\tilde \lambda_1))\cap L(\tilde \lambda_2)\not=\emptyset,$ which implies that  $\tilde f^n(\tilde \lambda_1)\cap \tilde \lambda_2 \not=\emptyset$.\end{proof}

\begin{proof}[Proof of Proposition  \ref{proposition:connections-better}] We consider the annuli:
 $$\hat U_1=\bigcup_{n\geq 0} \hat f^{-n}(L(\hat \lambda)),\enskip  \hat U_2=\bigcup_{n\geq 0} \hat f^{n}(R(\hat \lambda)), \enskip \hat U_{0}=  \hat U_1\cap \hat U_2,$$
 and the respective covering spaces:
  $$\tilde U_1=\bigcup_{n\geq 0} \tilde  f^{-n}(L(\tilde  \lambda_0)),\enskip  \tilde  U_2=\bigcup_{n\geq 0} \tilde  f^{n}(R(\tilde  \lambda_0)), \enskip\tilde  U_{0}= \tilde U_1\cap\tilde  U_2.$$
The three annuli are invariant by $\hat f$. By Caratheodory's Prime End Theory (see \cite{M} for instance), each annulus  $\hat U_0$, $\hat U_1$, $\hat U_{2}$, can be compactified as an annulus in such a way that the restriction of $\hat f$ to the former annulus extends to a homeomorphism of the compact annulus. More precisely, to compactify  $\hat U_0$,  one must add the circle of prime ends $\hat S_1$ corresponding to the end on the right of $\hat\lambda$ and  the circle of prime ends $\hat S_2$ corresponding to the end on the left of $\hat\lambda$; to compactify  $\hat U_1$,  one must add $\hat S_1$ and $\hat J_2$; to compactify  $\hat U_2$,  one must add $\hat J_1$ and $\hat S_2$. Furthermore, one can add the covering spaces $\tilde S_1$ and $\tilde S_2$ of $\hat S_1$ and $\hat S_2$ to $\tilde U_0$ in such a way that the restriction of $\tilde f$ extends continuously  to the added space. Similarly, one can add $\tilde S_1$, $\tilde J_2$ to $\tilde U_1$ and $\tilde J_1$, $\tilde S_2$  to $\tilde U_{2}$. For every $i\in\{0,1,2\}$, one can consider the sets $\tilde{\mathcal E}_1(\tilde U_i)$ and $\tilde{\mathcal E}_2(\tilde U_i)$ like in Sub-section \ref{s.forcing}, noting that $\tilde{\mathcal E}_1(\tilde U_0)$  is a subset of  $\tilde{\mathcal E}_1(\tilde U_1)$ and  $\tilde{\mathcal E}_2(\tilde U_0)$ a subset of $\tilde{\mathcal E}_2(\tilde U_2)$. One can define the rotation numbers $\rho_1$, $\rho_2$, defined respectively by $\tilde f$ on each spaces  $\tilde S_1$ and $\tilde S_2$. Note that at least one of the following conditions is satisfied:
  $$\rho_1\not=r_{i_0}, \enskip \rho_2\not=0, \enskip \rho_1\not=\rho_2,$$ meaning that a boundary twist condition is satisfied on at least one annulus $\hat U_0$, $\hat U_1$ or $\hat U_{2}$. The sets $\tilde U_0\cap \left(\bigcup_{\tilde \lambda\in\tilde{\mathcal L}_1}  \tilde\lambda\right)$ and $\tilde U_1\cap \left(\bigcup_{\tilde \lambda\in\tilde{\mathcal L}_1}  \tilde\lambda\right)$ coincide, we note $\tilde{\mathcal D}_1$ the set of its connected components. Each element $\tilde\delta_1$ of $\tilde{\mathcal D}_1$ is contained  in a line 
 $\tilde \lambda_1\in\tilde{\mathcal L}_1 $, moreover it holds that $\tilde \lambda_1\in\tilde{\mathcal L}_{1,l} $. Note that $ \tilde \delta_1$ is a line of  $\tilde U_1$ that has two different limit points in the frontier of $\tilde U_1$ in $\overline{\tilde S}$ (called accessible points). An important property of prime end theory is that  $ \tilde \lambda_1$ has two different limit points in $\tilde S_1$.  In particular  $\tilde{\mathcal D}_1$ is a subset of $\tilde{\mathcal E}_1(\tilde U_0)$ (and consequently of $\tilde{\mathcal E}_1(\tilde U_1)$). Moreover $T_0$ naturally acts on it. Note that for every $k\not=0$, every $\delta_1\in\tilde{\mathcal D}_1$,
 and every $\tilde \delta_1'\in\tilde{\mathcal D}_1$, one has $\tilde f^k(\delta_1)\cap \delta'_1=\emptyset$. Indeed, we have a similar properties in $\tilde{\mathcal L}_{1,l}$.  This means that one gets a subset $\bigcup_{k\in\Z}\tilde f^k( \tilde{\mathcal D}_1)$  of $\tilde{\mathcal E}_1(\tilde U_0)$ invariant by $T_0$ and $\tilde f$ that consists of pairwise disjoint lines. Consequently, $\prec$ is transitive on $\bigcup_{k\in\Z} \tilde f^k(\tilde{\mathcal D}_1)$ and induces an order $\preceq$ which is not necessarily total. Of course $T_0$ and $\tilde f$ preserve the order and  one has $ \tilde\delta \prec T_0(\tilde\delta)$ for every $\tilde\delta\in\bigcup_{k\in\Z} \tilde f^k(\tilde{\mathcal D}_1)$.
We can define similarly $\tilde{\mathcal D}_2\subset\tilde{\mathcal E}_2(\tilde U_0)$ with the same properties.

By Lemma \ref{lemma:connections}, there exist $\tilde \lambda_1\in\tilde{\mathcal L}_{1,l}$, $\tilde \lambda_2\in\tilde{\mathcal L}_{2,r}$ and $n_0\geq 1$ such that $\tilde f^{n_0}(\tilde \lambda_1)\cap \tilde \lambda_2\not=\emptyset$. Choose $\tilde z\in \tilde \lambda_1\cap \tilde f^{-n_0}(\tilde \lambda_2 )$. Its orbit is included in $\tilde U_0$. So, there exists  $\tilde\delta_1\in \tilde {\mathcal D}_1$ and  $\tilde\delta_2\in \tilde {\mathcal D}_2$ such that $\tilde z \in \tilde\delta_1$ and $\tilde f^{n_0}(\tilde z)\in\tilde\delta_2$.

We will begin by studying the case where $ \rho_1\not=r_{i_0}$  and will prove that the first item of Proposition \ref{proposition:connections-better} is satisfied. Then we will study the case where $ \rho_2\not=0$ and will prove that the second item is satisfied. Eventually we will look at the case where $\rho_1=r_{i_0}$ and $ \rho_2=0$  and will see that both items are satisfied.

Suppose first that $ \rho_1\not=r_{i_0}$.  Note that all properties of Proposition  \ref{prop:connexions-general} are satisfied with $\tilde\delta_1\in \tilde{\mathcal E}_1(\tilde U_1)$ and $\tilde\lambda_2\in\tilde{\mathcal E}_2(\tilde U_1)$. Setting $\rho_-=\min(\rho_1,r_{i_0})$ and $\rho_+=\max(\rho_1,r_{i_0})$ one gets that
$$\rho_-(n_0+q)<p<\rho_+(n_0+q)\Rightarrow\tilde f^{2n_0+q}(\tilde \delta_1)\cap T_0^p(\tilde\lambda_2)\not=\emptyset.$$
So, if $\rho_-(n_0+q)<p<p'<\rho_+(n_0+q)$, there exists a segment $\tilde\sigma_1\subset \tilde\delta_1\subset\tilde\lambda_1$ such that:
\begin{itemize}
\item $\tilde f^{2n_0+q}(\tilde \sigma_1)$ joins $T_0^{p}(\tilde\lambda_2)$ and $T_0^{p'}(\tilde\lambda_2)$;
\item the interior of  $\tilde f^{2n_0+q}(\tilde\sigma_1)$ is included in $R(T_0^{p}(\tilde\lambda_2))\cap R(T_0^{p'}(\tilde\lambda_2))$;
\item  $\tilde f^{2n_0+q}(\tilde\sigma_1)$ is included in $\bigcup_{ k\leq 0}\tilde f^{-k}(L(\tilde\lambda_0))$.

\end{itemize}

Suppose now that $ \rho_2\not=0$.  The properties of Proposition  \ref{prop:connexions-general} are satisfied with $\tilde\lambda_1\in \tilde{\mathcal E}_1(\tilde U_2)$ and $\tilde\delta_2\in\tilde{\mathcal E}_2(\tilde U_2)$. Setting $\rho_-=\min(\rho_2,0)$ and $\rho_+=\max(\rho_2,0)$ one gets that
$$\rho_-(n_0+q)<p<\rho_+(n_0+q)\Rightarrow\tilde f^{2n_0+q}(\tilde \lambda_1)\cap T_0^p(\tilde\delta_2)\not=\emptyset.$$
So, if $\rho_-(n_0+q)<p<p'<\rho_+(n_0+q),$ there exists a segment $\tilde\sigma_2\subset \tilde\delta_2\subset\tilde\lambda_2$ such that:
\begin{itemize}
\item $\tilde f^{-2n_0-q}(\tilde \sigma_2)$ joins $T_0^{-p}(\tilde\lambda_1)$ and $T_0^{-p'}(\tilde\lambda_1)$;
\item the interior of  $\tilde f^{-2n_0-q}(\tilde\sigma_2)$ is included in $L(T_0^{-p}(\tilde\lambda_1))\cap L(T_0^{-p'}(\tilde\lambda_1))$;
\item  $\tilde f^{-2n_0-q}(\tilde\sigma_2)$ is included in $\bigcup_{ k\leq 0}\tilde f^{-k}(R(\tilde\lambda_0))$.

\end{itemize}

Suppose now that $\rho_1=r_{i_0}$ and $\rho_2=0$.  The properties of Proposition  \ref{prop:connexions-general} are satisfied with $\tilde\delta_1\in \tilde{\mathcal E}_1(\tilde U_0)$ and $\tilde\delta_2\in\tilde{\mathcal E}_2(\tilde U_0 )$. Setting $\rho_-=\min(0,r_{i_0})$ and $\rho_+=\max(0,r_{i_0})$ one gets that
$$\rho_-(n_0+q)<p<\rho_+(n_0+q)\Rightarrow\tilde f^{2n_0+q}(\tilde \delta_1)\cap T_0^p(\tilde\delta_2)\not=\emptyset.$$

So, if $\rho_-(n_0+q)<p<p'<\rho_+(n_0+q), $ there exists a segment $\tilde\sigma_1\subset \tilde\delta_1\subset\tilde\lambda_1$ such that:
\begin{itemize}
\item $\tilde f^{2n_0+q}(\tilde \sigma_1)$ joins $T_0^{p}(\tilde\lambda_2)$ and $T_0^{p'}(\tilde\lambda_2)$;
\item the interior of  $\tilde f^{2n_0+q}(\tilde\sigma_1)$ is included in $R(T_0^{p}(\tilde\delta_2))\cap R(T_0^{p'}(\tilde\delta_2))$.
\item $\tilde f^{2n_0+q}(\tilde\sigma_1)$ is included in $\bigcup_{ k\leq 0}\tilde f^{-k}(L(\tilde\lambda_0))$.

\end{itemize}

Similarly, there exists a segment $\tilde\sigma_2\subset \tilde\delta_2\subset\tilde\lambda_2$ such that:
\begin{itemize}
\item $\tilde f^{-2n_0-q}(\tilde \sigma_2)$ joins $T_0^{-p}(\tilde\lambda_1)$ and $T_0^{-p'}(\tilde\lambda_1)$;

\item the interior of  $\tilde f^{-2n_0-q}(\tilde\sigma_2)$ is included in $L(T_0^{-p}(\tilde\lambda_1))\cap L(T_0^{-p'}(\tilde\lambda_1))$.
\item  $\tilde f^{-2n_0-q}(\tilde\sigma_2)$ is included in $\bigcup_{ k\leq 0}\tilde f^{-k}(R(\tilde\lambda_0))$.

\end{itemize}

\end{proof}

\begin{remark*}

 The boundary twist condition is satisfied on the whole space $\tilde S$. Setting $\rho_-=\min(0,r_{i_0})$ and $\rho_+=\max(0,r_{i_0})$ one knows that
 $$\rho_-(n_0+q)<p<\rho_+(n_0+q)\Rightarrow\tilde f^{2n_0+q}(\tilde \lambda_1)\cap T_0^p(\tilde\lambda_2)\not=\emptyset.$$
So, if $\rho_-(n_0+q)<p<p'<\rho_+(n_0+q), $ there exists a segment $\tilde\sigma_1\subset \tilde\lambda_1$ such that:
\begin{itemize}
\item $\tilde f^{2n_0+q}(\tilde \sigma_1)$ joins $T_0^{p}(\tilde\lambda_2)$ and $T_0^{p'}(\tilde\lambda_2)$;
\item the interior of  $\tilde f^{2n_0+q}(\tilde\sigma_1)$ is included in $R(T_0^{p}(\tilde\lambda_2))\cap R(T_0^{p'}(\tilde\lambda_2))$.
\end{itemize}
The problem is that we need the supplementary condition
\begin{itemize}
\item  $\tilde f^{2n_0+q}(\tilde\sigma_1)$ is included in $\bigcup_{ k\leq 0}\tilde f^{-k}(L(\tilde\lambda_0))$;
\end{itemize}
to make right the argument of the proof of Proposition \ref{prop:periodic}. There is no reason why such a condition will be satisfied even if $q$ is large enough. This is why we must work in $\tilde U_0$, $\tilde U_1$ or in $\tilde U_2$.
 
\end{remark*}

\section{Homeomorphisms isotopic to the identity}\label{s.isoidentty}

The goal of this section is to prove Proposition  \ref{prop:connexions}, which means to prove that  if $S$ is an orientable closed surface of genus $g\geq2$ and $f$ a non wandering homeomorphism of $S$  isotopic to the identity, then either $f$ has periodic points of period arbitrarily large, or every periodic point of $f$ is fixed and $f$ is isotopic to the identity relative to its fixed point set (the existence of at least one fixed point being a consequence of Lefschetz formula).

\begin{proof} [Proof of Proposition    \ref{prop:connexions}] We keep the same notations as before. The map $f$ being isotopic to the identity and the genus of $S$ being larger than $1$, there exists a unique lift  $\tilde f$ of $f$ to $\tilde S$ that commutes with the covering automorphims. A periodic point of $f$ that is lifted to a periodic point of $\tilde f$ is called a {\it contractible periodic point}. It was proved in \cite{Lec2} that:

\begin{itemize}
\item either $f$ has periodic points of arbitrarily large period;

\item or every contractible periodic point of $f$ is fixed and $f$ is isotopic to the identity relative to the contractible fixed point set. 
\end{itemize}
So, Proposition  \ref{prop:connexions} is an extension of this result and to obtain Proposition  \ref{prop:connexions} it remains to show  that every periodic point is contractible.  We will argue by contradiction and suppose that there exists a non contractible periodic point $z_0$, denoting by $q_0$ its period. If $\tilde z_0$ is a lift of $z_0$, there exists  $T_0\in G\setminus\{\mathrm{Id})$ such that $\tilde f^q(\tilde z_0)=T_0(\tilde z_0)$.  The map $\tilde f$ commutes with $T_0$ and lifts a homeomorphism $\hat f$ of $\hat S=\tilde S/T_0$. Recall that $\tilde f$ extends to a homeomorphism of $\overline{ \tilde S}$ that fixes every point of $\partial S$. Consequently $\hat f$ extends to a homeomorphism of the compact annulus $\overline{\hat S}$ obtained by adding the two circles $\hat J_1= \tilde J_1/T_0$ and $\hat J_2= \tilde J_2/T_0$, where $\tilde J_1$ and $\tilde J_2$ are the two connected components of $\partial S\setminus\{\alpha(T_0), \omega(T_0)\}$, the first one on the right of every $T_0$-line, the second on the left.  Note that every point of $\hat J_1\cup \hat J_2$ is fixed, with a rotation number equal to zero for the lift $\tilde f$. Note also that $\tilde z_0$ projected onto a periodic point $\hat z_0$ of period $q_0$ and rotation number $1/q_0$.  One can apply Theorem \ref{th:PB}:

\begin{itemize} 
\item either, for every rational number $p/q$ between $0$ and $1/q_0$, written in an irreducible way, there exists a periodic point $z$ of $\hat f$ of period $q$ and rotation number $p/q$ for $\tilde f$;
\item or there exists an essential simple  loop $\hat \lambda\subset \hat S$ such that $\hat f(\hat \lambda)\cap \hat \lambda=\emptyset.$ 
\end{itemize}  
In case the first item holds, we can prove, as we did in Section \ref{s.dehn} for a map isotopic to a Dehn twist map, that $f$ has periodic points of period arbitrarily large.  So, we can assume that the second item holds. Consider the lift $\tilde\lambda_0\subset \tilde S$ of $\hat\lambda$, oriented  as a $T_0$-line. Replacing $f$ with $f^{-1}$ if necessary, we can suppose that it is a Brouwer line. The fact that $\tilde f$ commutes with the covering automorphisms implies that every line $T(\tilde \lambda_0)$, $T\in G$,  is a Brouwer line of $\tilde f$. 

There is a natural partition 
$\tilde{\mathcal L}= \tilde{\mathcal L_0}\cup \tilde{\mathcal L_1}\cup \tilde {\mathcal L_2}$
of the set of lifts of $\hat\lambda$, defined as follows:
\begin{itemize}
\item every $\tilde\lambda\in  \tilde{\mathcal L_0}$ meets $ \tilde\lambda_0$;
\item every $\tilde\lambda\in  \tilde{\mathcal L_1}$ is included in $R(\tilde\lambda_0)$;
\item every $\tilde\lambda\in  \tilde{\mathcal L_2}$ is included $L(\tilde\lambda_0)$.
\end{itemize} 
Note that the ends of a line in $\tilde{\mathcal L_1}$ belong to $\tilde J_1$ and the ends of a line in $\tilde{\mathcal L_2}$ belong to $\tilde J_2$. Moreover, we have two partitions $\tilde{\mathcal L_1}=\tilde{\mathcal L}_{1,r}\cup \tilde {\mathcal L}_{1,l}$
and $\tilde{\mathcal L_2}=\tilde{\mathcal L}_{2,r}\cup \tilde {\mathcal L}_{2,l}$, where:
\begin{itemize}
\item $\tilde\lambda\in \tilde{\mathcal L}_{1,r}\cup \tilde{\mathcal L}_{2,r}$ if $\tilde\lambda_0\subset  R(\tilde\lambda)$;
\item $\tilde\lambda\in \tilde{\mathcal L}_{1,l}\cup \tilde{\mathcal L}_{2,l}$ if $\tilde\lambda_0\subset L(\tilde\lambda)$.
\end{itemize}

Note that the subsets defined above are all invariant by $T_0$. A last important remark is the following: there exists $N>0$ such that
if $\tilde \lambda$ is a lift of $\lambda$, there exists at most $N$ other lifts that meet $\tilde \lambda$ up to the action of $T$, where $T$ is the generator of the stabilizer of $\tilde \lambda$ (in particular the number of $T_0$-orbits in $\tilde{\mathcal L_0}$ is bounded by $N$). Indeed fix a segment $\tilde\delta\subset\tilde\lambda_0$ joining a point $\tilde z$ to $T_0(\tilde z)$ the set of $T\in G$ such that $T(\tilde\delta)\cap \tilde\delta\not=\emptyset$ is finite, choose $N$ to be its cardinal.

\begin{lemma} \label{l:connexion2} There exist $\tilde\lambda_1\in \tilde{\mathcal L}_{1,l}$, $\tilde\lambda_2\in \tilde{\mathcal L}_{1,r}$ and $n_0\geq 1$ such that $\tilde f^{n_0}(\tilde\lambda_1)\cap \tilde\lambda_2\not=\emptyset.$

\end{lemma}

\begin{proof} Consider an open disk $\tilde U\subset R(\tilde\lambda_0)\cap \tilde f^{-1}(L(\tilde \lambda_0))$ that projects onto an open disk $U$ of $S$.  Using Lemma \ref{l:NWH},  there exists $z\in U$ and two increasing sequences $(m_i)_{0\leq k\leq N}$, $(m'_i)_{i\leq k\leq N}$, with $m_0=m'_0=0$,  such that for every $k\in\{0,\dots,N\}$, the points $f^{m_k}(z)$ and $f^{-m'_k}(z)$ belong to $U$. Let $\tilde z$ be the lift of $z$ that belongs to $\tilde U$. So there exist two sequences $(T^{(k)})_{0\leq k\leq N}$, $(T'{}^{(k)})_{0\leq k\leq N}$ of covering automorphisms such that  $\tilde f^{m_k}(\tilde z)\in T^{(k)}(\tilde U)$ and $\tilde f^{ -m'_k}(\tilde z)\in T'{}^{(k)}(\tilde U)$.  Suppose that there exist $0\leq k<k'\leq N$ and $m\in\Z$ such that $T^{(k')}=T_0^m T^{(k)}$. Then it holds that $\tilde f^{m_k}(T^{(k)}{}^{-1}(z))\in \tilde U$ and $\tilde f^{m_{k'}}(T^{(k)}{}^{-1}(z))\in T_0^m(\tilde U)$.  We have a contradiction because $T_0^m(\tilde U)\subset  R(\tilde\lambda_0)$ and $\tilde f^{m_{k'}-m_k}( \tilde U)\subset L(\tilde\lambda_0)$. But we have $T^{(0)}=\mathrm{Id}$, and so, there exists $1\leq k\leq N$ such that $T^{(k)}(\tilde\lambda_0)\not\in  \tilde{\mathcal L}_0$. The lines $\tilde\lambda_0$ and $T^{(k)}(\tilde\lambda_0)$ being Brouwer lines, one deduces that  $T^{(k)}(\tilde\lambda_0)\in  \tilde{\mathcal L}_{2,r}$. For the same reasons, one can find $1\leq k'\leq N$ such that $T'{}^{(k')}(\tilde\lambda_0)\in  \tilde{\mathcal L}_{1,l}$.  Setting $n_0=m_k+m'_{k'}+1$, one deduce  that $\tilde f^{n_0}(R(T'{}^{(k')}(\tilde\lambda_0))\cap L(T^{(k)}(\tilde\lambda_0))\not=\emptyset$ which easily implies that $\tilde f^{n_0}(T'{}^{(k')}(\tilde\lambda_0))\cap T^{(k)}(\tilde\lambda_0)\not=\emptyset.$ 

\end{proof}
 We would like to give a proof similar to the proof in Section \ref{s.dehn}.  The fact that $\tilde z_0$ lifts a periodic point of $\hat f$ implies that the points $\tilde f^q\circ T_0^p(\tilde z_0)$, $p\in\Z$, $q\in \Z$, are all on the same side of $\tilde\lambda_0$. We will suppose that they are on the left side, meaning that they all belong to $\tilde U_1=\bigcup_{n\geq 0} \tilde  f^{-n}(L(\tilde  \lambda_0))$, the covering space of $\hat U_1=\bigcup_{n\geq 0} \hat f^{-n}(L(\hat \lambda_0))$. The case where they are on the right side can be treated similarly, replacing $\tilde U_1$ and $\hat U_1$ with  $\tilde U_2=\bigcup_{n\geq 0} \tilde  f^{n}(R(\tilde  \lambda_0))$ and  $\hat U_2=\bigcup_{n\geq 0} \hat f^{n}(R(\hat \lambda))$. Here again, we compactify  $\hat U_1$ by adding $\hat J_2$ and $\hat S^1$, the circle of prime ends corresponding to the end on the right of $\hat\lambda$. To obtain the universal covering space, we add $\tilde J_2$ and $\tilde S_1$, the covering space of $\hat S_1$, to $\tilde U_1$. The map $\tilde f_{\vert \tilde U_1}$ extends continuously  to the added lines and fixes every point of $\tilde J_2$. We denote $\rho_1$ the rotation numbers of $\widetilde f_{\vert\tilde S_1}$. Here again we can define the set  $\tilde{\mathcal D}_1$ of connected components of $\tilde U_1\cap \left(\bigcup_{\tilde \lambda\in\tilde{\mathcal L}_1}  \tilde\lambda\right)$, noting that every element $\tilde\delta$ is a line of $\tilde U_1$ that is contained in a line $\tilde\lambda\in \tilde{\mathcal L}_{1,l}$. 
 
 If we want to repeat the arguments given in Section  \ref{s.dehn}, we will meet a problem. In Section 3, the lines of  $\tilde\lambda\in  \tilde{\mathcal L_1}$ or  $\tilde\lambda\in  \tilde{\mathcal L_2}$ were pairwise disjoint, meaning that $\prec$ induces an order on these sets. This fact was very important in the proof because it was necessary for applying 
Proposition \ref{prop:connexions-general}. It is no more the case here. Nevertheless, there is at most $N$ lines  $T_0^k(\tilde\lambda_1)$, $k\in\Z$,  that intersect $\tilde\lambda_1$ and at most $N$ lines $T^k(\tilde\lambda_2)$ that intersect $\tilde\lambda_2$.  In particular if $s$ is large enough,  $\prec$ induces an order on the sets 
 $$\tilde{\mathcal L}'_{1,l}=\{T_0^{sk}(\tilde\lambda_1),\, k\in\Z\}\enskip \mathrm{and}\enskip \tilde{\mathcal L}'_{2,l}=\{T_0^{sk}(\tilde\lambda_2),\, k\in\Z\}.$$
 So we will have to work in the annulus $\tilde S/T_0^s$ instead of the annulus $\tilde S/T_0$. We define the set  $\tilde{\mathcal D}'_1\subset\tilde{\mathcal D}_1$ of connected components of $\tilde U_1\cap \left(\bigcup_{k\in\Z} T_0^{sk}(\tilde\lambda_1)\right)$.

 Like in Section \ref{s.dehn}, it holds that
$$\begin{aligned}
\tilde\lambda_1\in \tilde{\mathcal L}'_{1,l},\enskip \tilde\lambda'_1\in \tilde{\mathcal L}'_{1,l},\enskip k\in\Z\setminus\{0\}\Rightarrow \tilde f^k(\tilde\lambda_1)\cap \tilde\lambda'_1=\emptyset,\\
\tilde\delta_1\in \tilde{\mathcal D}'_{1},\enskip \tilde\delta'_1\in \tilde{\mathcal D}'_{1},\enskip k\in\Z\setminus\{0\}\Rightarrow \tilde f^k(\tilde\delta_1)\cap \tilde\delta'_1=\emptyset,\\
\tilde\lambda_2\in \tilde{\mathcal L}'_{2,r},\enskip \tilde\lambda'_2\in \tilde{\mathcal L}_{2,r}',\enskip k\in\Z\setminus\{0\}\Rightarrow \tilde f^k(\tilde\lambda_2)\cap \tilde\lambda'_2=\emptyset.
\end{aligned}$$
 Moreover,  as a consequence of Lemma  \ref{l:connexion2}, there exists $\tilde\delta_1\in \tilde{\mathcal D}'_{1}$ such that $\tilde f^{n_0}(\tilde\delta_1)\cap \tilde\lambda_2\not=\emptyset.$ 

There is two cases to study: the case where $\rho_1\not=0$ and the case where $\rho_1=0$. 

\bigskip

\noindent{\it Case where $\rho_1\not=0$.}  It is the case where there is a boundary twist condition. The rotation numbers in the new annulus are $\rho_1/s$ and $0$.  All the arguments appearing in Section \ref{s.dehn} are still valid. Setting $\rho_-=\min(0,\rho_1)$ and $\rho_+=\max(0,\rho_1)$ one gets that
$$\rho_-(n_0+q)<sp<\rho_+(n_0+q)\Rightarrow\tilde f^{2n_0+q}(\tilde \delta_1)\cap T_0^{sp}(\tilde\lambda_2)\not=\emptyset.$$

Moreover, if $$\rho_-(n_0+q)<sp<sp'<\rho_+(n_0+q), $$  then there exists a segment $\tilde\sigma_1\subset \tilde\delta_1\subset\tilde\lambda_1$ such that:
\begin{itemize}
\item $\tilde f^{2n_0+q}(\tilde \sigma_1)$ joins $T_0^{sp}(\tilde\lambda_2)$ and $T_0^{sp'}(\tilde\lambda_2)$;
\item the interior of  $\tilde f^{2n_0+q}(\tilde\sigma_1)$ is included in $R(T_0^{sp}(\tilde\lambda_2))\cap R(T_0^{sp'}(\tilde\lambda_2))$;
\item  $\tilde f^{2n_0+q}(\tilde\sigma_1)$ is included in $\bigcup_{ k\leq 0}\tilde f^{-k}(L(\tilde\lambda_0))$.

\end{itemize}

Like in Section \ref{s.dehn}, we deduce that for every $s\geq 2$, there exists $m_s\geq 0$ such that for every $m\geq m_s$, there exists $\tilde \lambda'_2\in \tilde{\mathcal L}_{2,r}$ such that for every $0<p<p'\leq s$, there exists a segment $\tilde\sigma_1\subset \tilde\lambda_1$  satisfying:
\begin{itemize}
\item $\tilde f^{m}(\tilde \sigma_1)$ joins $T_0^{sp}(\tilde\lambda'_2)$ and $T_0^{sp'}(\tilde\lambda'_2)$;
\item the interior of  $\tilde f^{m}(\tilde\sigma_1)$ is included in $R(T_0^{sp}(\tilde\lambda'_2))\cap R(T_0^{sp'}(\tilde\lambda'_2))$;
\item  $\tilde f^{m}(\tilde\sigma_1)$ is included in $L(\tilde\lambda_0)$;
\end{itemize}
  
 Fix $T_1\in G$ that sends $ \tilde\lambda_1$ onto  $\tilde\lambda_0$. Like in Section \ref{s.dehn}, we can prove that for every $m\geq m_5$, there exists $T_m\in G$ such that $\tilde f^m\circ T_m^{-1}$ has a fixed point, where $T_m$ can be written $T_m =  T_0^{sn_m}\circ T_1$. 
Choose a fixed point $\tilde z_m$ of $\tilde f^{m}\circ T_m^{-1}$. It projects onto a fixed point $z_m\in S$ of $f^m$. Let us prove that the period of $z_m$ tends to $+\infty$ with $m$. Otherwise, there exists $r\geq 0$ and 
an increasing sequence $(m_l)_{l\geq 0}$ such that $z_{m_l}$ has period $r$. So, there exists $S_l\in G$ such that $\tilde f^{r}(\tilde z_{m_l})=S_l(\tilde z_{m_l})$. The map $\tilde f$ commutes with the covering automorphisms. We deduce on one side that $\tilde f^{m_l}(\tilde z_{m_l})=T_{m_l}(\tilde z_{m_l})$ and on the other side that $\tilde f^{m_l}(\tilde z_{m_l})=S_l^{m_l/ r}(\tilde z_{m_l})$ and so $T_{m_l}=S_l^{m_l/ r}$.  Let us explain why it is impossible if $l$ is large enough. Note that 
$$S_l^{m_l/ r}(\tilde\lambda_1)=T_{m_l}( \tilde \lambda_1)=  T_0^{sn_m}\circ T_1( \tilde \lambda_1)=  T_0^{sn_m} (\tilde \lambda_0)=  \tilde \lambda_0.$$ This implies that
$$S_l^{m_l/ r}(\alpha(\tilde\lambda_1))=\alpha(\tilde\lambda_0) \enskip\mathrm{and}\enskip S_l^{m_l/ r}(\omega(\tilde\lambda_1))=\omega(\tilde\lambda_0) .$$
One can find two disjoint segments $\tilde\sigma_{\alpha}$ and $\tilde\sigma_{\omega}$ of $\partial \tilde S$, the first one joining $\alpha(\tilde\lambda_1)$ to $\alpha(\tilde\lambda_0)$, the second one joining $\omega(\tilde\lambda_1) $ to $\omega(\tilde\lambda_0) $. This implies that  for every $0\leq k\leq m_l/r$ it holds that $S_l^k(\alpha(\tilde\lambda_1))\in \sigma_{\alpha}$ and   $S_l^k(\omega(\tilde\lambda_1))\in \sigma_{\omega}$.  Let $\tilde\sigma$ be a segment of $\tilde S$ that joins $\tilde\lambda_1$ to $\tilde\lambda_2$. The set of lifts of $\hat\lambda$ that meet $\tilde\sigma$ is finite. Let $N'$ be its cardinal. Suppose that  $m_l/r\geq 2N+N'$. There exists $0< k<m_l/r$ such that $S_l^k(\tilde \lambda_1)$ does not meet $\tilde \lambda_1\cup\tilde \lambda _0\cup\tilde\sigma$. Nevertheless one of the end of  $S_l^k(\tilde \lambda_1)$ is in the interior of $\tilde\sigma_{\alpha}$ and the other one in the interior of $\tilde\sigma_{\omega}$. We have a contradiction.

\bigskip

\noindent{\it Case where $\rho_1=0$.}  Here there is no boundary twist condition. The twist condition is given by the existence of a periodic point which has a non zero rotation number.  The proof is inspired by arguments of Lellouch appearing in his thesis \cite{Lel}.

\begin{lemma} \label{l:connexion3} There exists an increasing sequence $(n_p)_{p\geq 0} $ such that  $\tilde f^n(\tilde\delta_1)\cap T_0^{sp}(\tilde\lambda_2)\not=\emptyset$
 if $n\geq n_p$.

\end{lemma}

\begin{proof} Every lift of $\hat \lambda$ being a Brouwer line, it is sufficient to prove that there exists $n_p$ such that $\tilde f^{n_p}(\tilde\delta_1)\cap T_0^{sp}(\tilde\lambda_2)\not=\emptyset$. Indeed, if $n>n_p$, then $\tilde f^{n}(\tilde\delta_1)\cap L(T_0^{sp}(\tilde\lambda_2))\not=\emptyset$ because
$$\tilde f^{n-n_p}(\tilde f^{n_p}(\tilde\delta_1)\cap T_0^{sp}(\tilde\lambda_2))=\tilde f^{n}(\tilde\delta_1)\cap \tilde f^{n-n_p}(T_0^{sp}(\tilde\lambda_2))\subset \tilde f^{n}(\tilde\delta_1)\cap L(T_0^{sp}(\tilde\lambda_2)),$$ and it implies that $\tilde f^{n_p}(\tilde\delta_1)\cap T_0^{sp}(\tilde\lambda_2)\not=\emptyset$. By induction, it is sufficient to prove that there exists $n_1\geq n_0$ such that  $\tilde f^{n_1}(\tilde\delta_1)\cap T^s_0(\tilde\lambda_2)\not=\emptyset$.

The fact that  $\tilde f^{n_0}(\tilde\delta_1)\cap \tilde\lambda_2\not=\emptyset$ implies that there exists a half-line $\tilde l_1\subset\tilde\delta_1$ and a half-line $\tilde l_2\subset\tilde\lambda_2$ such that
$\tilde f^{n_0}(\tilde l_1)$ and $\tilde l_2$ intersect in a unique point and such that $\tilde f^{n_0}(\tilde l_1)\cup\tilde l_2$ is a line $\tilde l$ of $\tilde U_1$. For the same reason, there exists a half-line $\tilde l'_1\subset\tilde\delta_1$ and a half-line $\tilde l'_2\subset\tilde\lambda_2$ such that
$\tilde l'_1$ and $\tilde f^{-n_0}(\tilde l'_2)$ intersect in a unique point and such that $\tilde l'_1\cup\tilde f^{-n_0}(\tilde l'_2)$ is a line $\tilde l'$ of $\tilde U_1$. We can also make a choice such that $\tilde l_1\cap\tilde l'_1$ is a half-line of $\tilde \delta_1$ and $\tilde l_2\cap\tilde l'_2$ a half-line of $\tilde \lambda_2$ 

If $m\geq 0$, then $\tilde f^m(\tilde l_2)$ and $\tilde f^{m+n_0}(\tilde l_2)$ are contained in $\overline{L(\tilde\lambda_2)}$ and so are disjoint from $T_0^s(\tilde l'_1)$ and $T_0^s( \tilde l'_2)$. We deduce that $\tilde f^m(\tilde l_2)\cap T_0^s(\tilde l')=\emptyset$. Moreover $\tilde f^m(\tilde l_1)\cap T_0^s(\tilde l'_1)=\emptyset$. So, to get Lemma \ref{l:connexion3}, it is sufficient to prove that there exists $m>0$ such that $\tilde f^m(\tilde l)\cap T_0^s(\tilde l')\not=\emptyset$. 

We will argue by contradiction and suppose that $\tilde f^m(\tilde l)\cap T^s_0(\tilde l')=\emptyset$ for every $m\geq 0$. We can orient $\tilde l$ and $T_0^s(\tilde l')$ such that $L(\tilde l)\subset L(T_0^s(\tilde l'))$. The ends of $\tilde l$ and $\tilde l'$ (on $\tilde S_1$ and $\tilde J_2$) are the same. The fact that $\rho_1=0$ implies that for every $m>0$, the ends of $\tilde f^m(\tilde l)$, which are the images by $\tilde f^m$ of the ends of $\tilde l$, stay smaller than the ends of $T_0^s(\tilde l')$, which are the images by $\tilde T_0^s$ of the ends of $\tilde l$. So we have $L(\tilde f^m(\tilde l))\subset L(T_0^s(\tilde l'))$. 
To get the contradiction, just notice that if $m$ is large enough, then $\tilde f^{-m}(\tilde z_0)\in  L(\tilde l)$ and $\tilde f^{m}(\tilde z_0)\in  R(T_0^s(\tilde l'))$. But we should have
$$\tilde f^{m}(\tilde z_0)=\tilde f^{2m}(\tilde f^{-m}(\tilde z_0))\in {L(\tilde f^{2m}(\tilde l))}\subset L(T_0^s(\tilde l')).$$

\end{proof}
Denote $\tilde \lambda_1$ the element of ${\mathcal L}_{1,l}$ that contains $\tilde\delta_1$. We deduce from Lemma  \ref{l:connexion3} that, for every $q\geq 2$, for every $m\geq n_q$ and every $0<p<p'\leq q$, there exists a segment $\tilde\sigma_1\subset \tilde\lambda_1$  satisfying:
\begin{itemize}
\item $\tilde f^{m}(\tilde \sigma_1)$ joins $T_0^{sp}(\tilde\lambda_2)$ and $T_0^{sp'}(\tilde\lambda_2)$;
\item the interior of  $\tilde f^{m}(\tilde\sigma_1)$ is included in $R(T_0^{sp}(\tilde\lambda_2))\cap R(T_0^{sp'}(\tilde\lambda_2))$;
\item  $\tilde f^{m}(\tilde\sigma_1)$ is included in $\bigcup_{ k\leq 0}\tilde f^{-k}(L(\tilde\lambda_0))$.
\end{itemize}

Like in the proof of  Proposition \ref{proposition:connections-better-better}, and using the the fact that the $T_0^{sp}(\tilde\lambda_2)$ are Brouwer lines (or equivalently that $[f]$ is the identity map) we deduce that, for every $q\geq 2$, there exists $m_q\geq n_q$ such that for every $m\geq m_q$ and every $0<p<p'\leq q$, there exists a segment $\tilde\sigma_1\subset \tilde\lambda_1$  satisfying:
\begin{itemize}
\item $\tilde f^{m}(\tilde \sigma_1)$ joins $T_0^{sp}(\tilde\lambda_2)$ and $T_0^{sp'}(\tilde\lambda_2)$;
\item the interior of  $\tilde f^{m}(\tilde\sigma_1)$ is included in $R(T_0^{sp}(\tilde\lambda_2))$ and in $R(T_0^{sp'}(\tilde\lambda_2))$;
\item  $\tilde f^{m}(\tilde\sigma_1)$ is included in $L(\tilde\lambda_0)$.

\end{itemize}
  
Finally, like in the proof of  Proposition \ref{prop:periodic},  we show that if $T_1$ is the unique covering automorphism  such that
$$T_1(\tilde\lambda_1)=\tilde\lambda_0\enskip\mathrm{and} \enskip T_0^{3s}(\tilde\lambda_2)\prec T_1(T^s_0(\tilde\lambda_1))\prec T_0^{4s}(\tilde\lambda_2),$$
then, for every $m\geq m_5$, the map  $\tilde f^{m}\circ T_1^{-1}$ has a fixed point.

To conclude, choose a fixed point $\tilde z_m$ of $\tilde f^{m}\circ T_1^{-1}$. It projects onto a fixed point $z_m\in S$ of $f^m$. Let us prove that the period of $z_m$ tends to $+\infty$ with $m$. Otherwise, there exist $r\geq 0$ and 
an increasing sequence $(m_l)_{l\geq 0}$ such that $z_{m_l}$ has period $r$. So, there exists $S_l\in G$ such that $\tilde f^{r}(\tilde z_{m_l})=S_l(\tilde z_{m_l})$. One deduces that  $\tilde f^{m_l}(\tilde z_{m_l})=S_l^{m_l/ r}(\tilde z_{m_l})$ and so $T_1=S_l^{m_l/ r}$. In particular $S_l$ belongs to the centralizer of $T$. But it is well known that the centralizer of $T_1$ is a cyclic group. We have got a contradiction.  \end{proof}

\section{The case of the torus}\label{s.torus}

The goal of this section is to prove Theorem \ref{th:rappel2}. The arguments that follow are the ones appearing in \cite{AT} but we need to verify that, up to slight modifications, they are still valid when the area-preserving condition is replaced with the non wandering condition.

Let us consider $M=\begin{pmatrix}a&b\\ c&d\end{pmatrix}\in \mathrm{SL}(2,\Z)$. There are three possibilities:

\begin{itemize}
\item $M$  is hyperbolic, meaning that its eigenvalues have a modulus different from $1$;
\item $M$ is conjugate in  $\mathrm{SL}(2,\Z)$ to $\begin{pmatrix}1&k\\ 0&1\end{pmatrix}$ or to $\begin{pmatrix}-1&k\\ 0&-1\end{pmatrix}$, where $k\in\Z\setminus\{0\}$;
\item $M$ has finite order (and in that case its order is $1$, $2$, $3$, $4$ or $6$).

\end{itemize}
The matrix $M$ induces an orientation preserving automorphism $[M]$ of $\T^2$ by the formula $[M](x,y)=(a x+by, cx+dy)$. We will denote $\mathrm{Aut}(\T^2)$ the group of such automorphisms. Every orientation preserving homeomorphism $f$ of $\T^2$ is isotopic to a unique automorphism, that will be denoted $[f]$, as its associated matrix.

Setting $D=\begin{pmatrix}1&1\\ 0&1\end{pmatrix}$, we have the following classification for  an orientation  homeomorphism  $f$ of $\T^2$ :

\begin{itemize}

\item $[f]$ is hyperbolic;

\item there exists $q\in\{1,2\}$ and $k\in\Z\setminus\{0\}$ such that $f^q$ is conjugate to a homeomorphism isotopic to $[D]^k$;

\item  there exists $q\in\{1,2,3,4,6\}$ such that $f^q$ is isotopic to the identity.
\end{itemize}

Let us recall now the definition of the rotation set of a homeomorphism of $\T^2$ isotopic to the identity (see \cite{MZ} or \cite{S}). For every homeomorphism $f$ of $\T^2$, denote ${\mathcal M}_f$ the set of Borel probability measures invariant by $f$.  Suppose that $f$ is isotopic to the identity. Every lift $\tilde f$ of $f$ to $\R^2$ commutes with the integer translations $\tilde z\mapsto z+k$, $k\in\Z^2$. So, the map $\tilde f-\mathrm{Id}$ lifts a continuous function $\psi_{\tilde f} : \T^2\to\R^2$. One can define the {\it rotation vector} $ \mathrm{rot}_{\tilde f} (\mu)=\int_{\T^2} \psi_{\tilde f} \, d\mu\in\R^2$ of  $\mu\in {\mathcal M}_f$, that measures the mean displacement of $\tilde f$.  The {\it rotation set}  $ \mathrm{rot}(\tilde f)=\{  \mathrm{rot}_{\tilde f}(\mu)\, \vert\, \mu\in{\mathcal M}_f\}$ is a non empty compact convex subset of $\R^2$. Of course, it depends on the lift $\tilde f$ but if $\tilde f'=  \tilde f+k$, $k\in\Z^2$, is another lift, it holds that $ \mathrm{rot}(\tilde f')=\mathrm{rot}(\tilde f)+k$ because 
 $\psi_{\tilde f'}=\psi_{\tilde f}+k$.

 The following properties are very classical (the two first ones are due to Franks \cite{F2}, \cite{F3}, the last one is an easy consequence of the characterization of the ergodic measures as extremal points of $\mathcal {M}_f$).
 
 \begin{enumerate}
 \item  If $p/q$ belongs to the interior of $ \mathrm{rot}(\tilde f)$, where $p\in\Z^2$ and $q\in\N\setminus\{0\}$, then there exists $\tilde z\in\R^2$ such that $\tilde f^q(\tilde z)=\tilde z+p$.
 \item  If $p\in\Z^2$ and $q\in\N\setminus\{0\}$ are such that $p/q= \mathrm{rot}_{\tilde f} (\mu)$, where $\mu\in {\mathcal M}_f$ is ergodic, then there exists $\tilde z\in\R^2$ such that $\tilde f^q(\tilde z)= \tilde z+p$.
 \item Every extremal point of $ \mathrm{rot}(\tilde f)$ is the rotation vector of an ergodic measure.
 
 \end{enumerate}

Let us recall now the definition of the vertical rotation set of a homeomorphism of $\T^2$ isotopic to a Dehn twist (see \cite{A} or \cite{Do}). Suppose that $f$ is isotopic to $[D]^k$, where $k\not=0$ and that $\hat f$ is a lift of $f$ to $\T\times \R$.  It commutes with the vertical translation $V: \hat z\mapsto \hat z+(0,1)$. So, the map $p_2\circ \hat f-p_2$ lifts a continuous function $\psi_{\hat f} : \T^2\to\R$. One can define the {\it vertical rotation number}  $  \mathrm{vrot}_{\hat f} (\mu)=\int_{\T^2} \psi_{\hat f} \, d\mu\in\R$ of a measure $\mu\in {\mathcal M}_f$ and the {\it vertical rotation set} $  \mathrm{vrot}(\hat f)=\{ \rho(\mu)\, \vert\, \mu\in{\mathcal M}_f\}$, which is a non empty segment of $\R$. Here again, it depends on the lift $\hat f$ but if $\hat  f'= V^k\circ \hat f$, $k\in\Z$, is another lift, we have $  \mathrm{vrot}(\hat f')= \mathrm{vrot}(\hat f)+k$. 
 
 We will need the following properties, the first one being proved in \cite{A} and \cite{Do}, the second one in \cite{AGT}:
 
 \begin{enumerate}
 \item if $p/q$ belongs to the interior of $  \mathrm{vrot}(\hat f)$, where $p\in\Z$ and $q\in\N\setminus\{0\}$, there exists $\hat z\in\T\times\R$ such that $\hat f^q(\hat z)= V^p(\hat z)$;
 \item   if $\mathrm{vrot}(\hat f))=\{p/q\}$, where $p\in\Z$ and $q\in\N\setminus\{0\}$, there exists a compact connected essential set $\hat K\subset \T\times\R$  invariant by $\hat f^q\circ V^{-p}$.  
 \end{enumerate}

Let $f$ be an orientation preserving homeomorphim of $\T^2$. It is well known that if $[f]$ is hyperbolic, then  $f$ has periodic points of period arbitrarily large.  So, Theorem \ref{th:rappel2} can be deduced from the two following results:

\begin{proposition} \label{prop:twist2} Let $f$ a non wandering homeomorphism of $\T^2$ isotopic to $[D]^k$, $k\not=0$. Then: 
\begin{itemize}
\item either $f$ has periodic points of period arbitrarily large;

\item or $f$ has no periodic orbit and there exists $\delta\in \T\setminus \Q/\Z$ such that for every lift $\hat f$ of $f$ to $\T\times\R$, there exists $
\hat \delta\in\R$ such that $\hat\delta+\Z=\delta$ and $\mathrm{vrot}(\hat f)=\{\hat\delta\}$.
\end{itemize} 
\end{proposition}

\begin{proposition} \label{prop:isid2} Let $f$ a non wandering homeomorphism of $\T^2$ isotopic to the identity. Exactly one of the following assertions holds:
\begin{enumerate}
\item $f$ has periodic points of period arbitrarily large;
\item If $\tilde f$ is a lift of $f$ to $\R^2$, then $\mathrm{rot}(\tilde f)$ is a point or a segment that does not meet $\Q^2/\Z^2$. In this case $f$ has no periodic point.
\item There exists an integer $q\geq 1$ such that:

\begin{itemize}
\item the periodic points of $f^q$ are fixed;
\item the fixed point set of $f^q$ is non empty and $f^q$ is isotopic to the identity relative to it;
\item the rotation set of the lift of $f^q$ that has fixed points is reduced to $0$ or is a segment with irrational slope that has zero as an end point.
\end{itemize}

\end{enumerate}

\end{proposition}

\begin{proof} [Proof of Proposition \ref{prop:twist2}] Fix a lift $\hat f$ of $f$ to $\T\times \R$. If $\mathrm{vrot}(\hat f)$ is not reduced to a point, then for every rational number $p/q\in \mathrm{int}(\mathrm{vrot}(\hat f))$ there exists $\hat z\in\T\times\R$ such that $\hat f^q(\hat z)= V^p(\hat z)$.  If $p$ and $q$ are chosen relatively prime, then  $\hat z$ projects onto a periodic point of $f$ of period $q$. Consequently, $f$ has periodic points of period arbitrarily large.

Suppose now that $\mathrm{vrot}(\hat f)$ is reduced to a rational number $p/q$. Replacing $f$ with $f^q$ and $\hat f$ with $V^{-p}\circ \hat f ^q$, one can suppose that  $\mathrm{vrot}(\hat f)=\{0\}$. There exists a compact connected essential set $\hat K\subset \T\times\R$ that is invariant by $\hat f$.  Let $\mu$ be a Borel probability measure supported on $\hat K$ and invariant by $\hat f$. For every $m\in \Z$, the measure $V^m_*(\mu)$ is supported on $V^m(\hat K)$ and invariant by $\hat f$. Fix a lift $\tilde f$ of $\hat f$ to $\R^2$. For every $m\in\Z$, one has $\mathrm{rot}_{\tilde f}( V_*^m(\mu))=\mathrm{rot}_{\tilde f} (\mu)+mk$. So, by Theorem \ref{th:PB} it holds that

\begin{itemize} 
\item either, for every rational number $p/q\in\R$, there exists a periodic point $\hat z$ of $\hat f$ of period $q$ and rotation number $p/q$ for $\tilde f$;
\item or there exists an essential simple loop $\hat\lambda\in \T\times\R$ such that $\hat f(\hat\lambda)\cap\hat\lambda=\emptyset$.

\end{itemize}
Here again, in the first situation, if $p$ and $q$ are chosen relatively prime, $\hat z$  projects onto a periodic point of $f$ of period $q$ and so $f$ has periodic points of period arbitrarily large. Let us prove now that the second situation never holds. Suppose that there exists an essential simple loop $\hat\lambda\in \T\times\R$ such that $\hat f(\hat \lambda)\cap\hat\lambda=\emptyset$. Then one can find a relatively compact wandering disk $\hat U$. The fact that every  $V^m(\hat K)$, $m\in\Z$, is compact, essential and $\hat f$-invariant implies that $\bigcup_{k\in\Z} \hat f^k(\hat U)$ is relatively compact.  This contradicts Proposition   \ref{prop:wandering}. \end{proof}

\begin{proof}  [Proof of Proposition \ref{prop:isid2}]  Fix a lift $\tilde f$ of $f$ to $\R^2$. If $\mathrm{rot}(\tilde f)$ has non empty interior, then $f$ has periodic points of period arbitrarily large. Indeed, if  $(p_1/q, p_2/q)$ belongs to the interior of $ \mathrm{rot}(\tilde f)$, then there exists $\tilde z\in\R^2$ such that $\tilde f^q(\tilde z)= \tilde z+(p_1,p_2)$. Moreover, if $p_1$, $p_2$ and $q$ are chosen with no common divisor, then $\tilde z$ projects onto a  periodic point of $f$ of period $q$.

 Suppose now that $\mathrm{rot}(\tilde f)$ is a point or a segment that does not meet $\Q^2/\Z^2$. Then $f$ has no periodic point and (2) holds.  
 
 Suppose now that $\mathrm{rot}(\tilde f)$ meets $\Q^2/\Z^2$ in a unique point $p/q$. Either $\mathrm{rot}(\tilde f)$ is reduced to $p/q$ or is a segment with irrational slope. It has been proven in \cite{LT1} that $p/q$ is an end point of $\mathrm{rot}(\tilde f)$ in this last case. In particular in both cases, $\tilde f^q-p$ has at least one fixed point, because $p/q$ is the rotation vector of an ergodic measure.  Note that every periodic point of $f^q$ is lifted to a periodic point of $\tilde f^q-p$, meaning it is contractible. Using \cite{Lec2}, one deduces that:
\begin{itemize}
\item either $f$ has periodic points of arbitrarily large period;
\item or the periodic points of $f^q$ are all fixed and lifted to fixed points of $ \tilde f^q-p$, moreover $f^q$ is isotopic to the identity relative to its fixed point set.
\end{itemize} 

It remains to study the case where $\mathrm{rot}(\tilde f)$ is  a non trivial segment with rational slope that intersects $\Q^2/\Z^2$. Replacing $f$ with $f^q$ and $\tilde f$ with $\tilde f^q-p$, where $p/q\in\mathrm{rot}(\tilde f)$, one can suppose that $0 \in \mathrm{rot}(\tilde f)$. The linear line containing $\mathrm{rot}(\tilde f)$ is generated by $p'\in\Z^2\setminus\{0\}$ and invariant by the translation $T: z\mapsto z+p'$.   Let $\hat f$ be the homeomorphism of the annulus  $A_{p'}=\R^2/T$ lifted by $\tilde f$.  A result of D\'avalos \cite{Da} tells us that the orbits of $\hat f$ are uniformly bounded, or equivalently that there exists a compact connected essential set $\hat K\subset \A_{p'}$  invariant by $\hat f$.  The rotation set of $\tilde f$ being non trivial, one can find two compactly supported ergodic measures of $\hat f$ with different rotation numbers (for $\tilde f$). Like in the proof of Proposition \ref{prop:twist2}, we can prove that for every essential simple loop $\hat\lambda\in \A_{p'}$ it holds that $\hat f(\hat\lambda)\cap\hat\lambda\not=\emptyset$ and then that $f$ has periodic point of arbitrarily large periods.  \end{proof}

\end{document}